\documentclass[11pt]{article}

\usepackage{stmaryrd}
\usepackage{natbib}
\usepackage{graphicx}

\setcitestyle{numbers,open={[},close={]}}

\usepackage[utf8]{inputenc}
\usepackage[T1]{fontenc}
\usepackage[english]{babel}
\usepackage{graphicx}
\usepackage{fullpage}
\usepackage{caption}
\usepackage{subcaption}
\usepackage{lipsum}

\usepackage{slashed}
\usepackage{amsmath}
\usepackage{amssymb}
\usepackage{amsthm}
\usepackage{latexsym}
\usepackage{color}
\usepackage{dsfont}
\usepackage{graphicx}
\usepackage{color}
\usepackage{hyperref}
\hypersetup{
    colorlinks,
    citecolor=blue,
    filecolor=blue,
    linkcolor=blue,
    urlcolor=blue
}

\DeclareSymbolFont{calletters}{OMS}{cmsy}{m}{n}
\DeclareSymbolFontAlphabet{\mathcal}{calletters}

\def\be{\begin{eqnarray}}
\def\ee{\end{eqnarray}}

\def\b*{\begin{eqnarray*}}
\def\e*{\end{eqnarray*}}

\newtheorem{Theorem}{Theorem}[section]
\newtheorem{Definition}[Theorem]{Definition}
\newtheorem{Proposition}[Theorem]{Proposition}

\newtheorem{Assumption}[Theorem]{Assumption}

\newtheorem{Lemma}[Theorem]{Lemma}
\newtheorem{Corollary}[Theorem]{Corollary}
\newtheorem{Remark}[Theorem]{Remark}

\makeatletter \@addtoreset{equation}{section}

\usepackage{color}

\def \E{\mathbb{E}}
\def \F{\mathbb{F}}
\def \H{\mathbb{H}}
\def \L{\mathbb{L}}
\def \P{\mathbb{P}}

\def \R{\mathbb{R}}
\def \S{\mathbb{S}}
\def \M{\mathbf{M}}
\def \N{\mathbb{N}}

\def\Ac{{\cal A}}
\def\Bc{{\cal B}}
\def\Cc{{\cal C}}

\def\Fc{{\cal F}}

\def\Jc{{\cal J}}

\def\Lc{{\cal L}}
\def\Mc{{\cal M}}

\def\Pc{{\cal P}}

\def\Sc{{\cal S}}

\def\Wc{{\cal W}}

\def\Omt{\tilde{\Om}}
\def\Fct{\tilde{\Fc}}

\def\Zt{\tilde{Z}}

\def \Om{\Omega}
\def \om{\omega}

\def \omt{\tilde{\om}}
\def \xb{\mathbf{x}}

\def \0{\mathbf{0}}
\def \Xh{\widehat{X}}

\def\x{\times}

\def\1{{\bf 1}}
\def \proof{{\noindent \bf Proof.\quad}}

\def \no{\noindent}
\def \ep{\hbox{ }\hfill{ ${\cal t}$~\hspace{-5.1mm}~${\cal u}$}}

\newcommand{\bea}{\begin{eqnarray}}
\newcommand{\bes}{\begin{subequations}}
\newcommand{\ees}{\end{subequations}}
\newcommand{\bgt}{\begin{gather}}
\newcommand{\egt}{\begin{gather}}
\newcommand{\eea}{\end{eqnarray}}
\newcommand{\beaa}{\begin{eqnarray*}}
\newcommand{\eeaa}{\end{eqnarray*}}

\def\x{\times}

\def\Om{\Omega}
\def\Omt{\widetilde{\Omega}}
\def\Omh{\widehat{\Omega}}
\def\om{\omega}

\def\Fct{\widetilde{\Fc}}
\def\Ft{\widetilde{\F}}

\def\Wt{\widetilde{W}}
\def\Xt{\widetilde{X}}

\def\0{\mathbf{0}}

\def \xb{\mathbf{x}}

\def \mub{\overline{\mu}}

\def \muh{\widehat{\mu}}

\def \mut{\widetilde{\mu}}

\def\normeL2#1{\left\|{#1}\right\|_{L^2}}

\def\Xh{\widehat X}

\def \Lim{\displaystyle\lim}

\def \Xbb{\mathbf{X}}

\def \Wbb{\mathbf{W}}

\def \Pr{\mathrm{P}}
\def \Qr{\mathrm{Q}}

\title{\bf Mean Field Game of Mutual Holding}

\author{Mao Fabrice Djete\thanks{Ecole Polytechnique Paris, Centre de Math\'ematiques Appliqu\'ees,
mao-fabrice.djete@polytechnique.edu. This work benefits from the financial support of the Chairs {\it Financial Risk} and {\it Finance and Sustainable Development}. }
        \and Nizar Touzi\thanks{Ecole Polytechnique Paris, Centre de Math\'ematiques Appliqu\'ees,
        nizar.touzi@polytechnique.edu. This work benefits from the financial support of the Chairs {\it Financial Risk} and {\it Finance and Sustainable Development}. }}

\date{\today}

\begin{document}



\maketitle

\begin{abstract}
We introduce a mean field model for optimal holding of a representative agent of her peers as a natural expected scaling limit from the corresponding $N-$agent model. The induced mean field dynamics appear naturally in a form which is not covered by standard McKean-Vlasov stochastic differential equations. We study the corresponding mean field game of mutual holding in the absence of common noise. {\color{black} Our first main result provides an explicit equilibrium of this mean field game, defined by a bang--bang control consisting in holding 
those competitors with positive drift coefficient of their dynamic value. We next use this mean field game equilibrium to construct (approximate) Nash equilibria for the corresponding $N$--player game. We also provide some numerical illustrations of our mean field game equilibrium which highlight some unexpected effects induced by our results.
}

\end{abstract}

\vspace{3mm}
\no{\bf Keywords.} Mean field McKean-Vlasov stochastic differential equation, mean field game, backward stochastic differential equations.

\no{\bf MSC2010.} 60K35, 60H30, 91A13, 91A23, 91B30.

\section{Introduction}

Connexions between economic actors lies at the heart of the foundations of financial theory as it offers the possibility of risk diversification. Such interconnexions result from mutual holding equity or debt obligations, and investment in common financial assets, thus illustrating the strategic behavior of economic actors in terms of risk sharing. The equity value is then determined endogenously, consistently with the rational behavior of agents. The value of these claims depends on the financial health of 
the obligor who, himself, might be an obligee, and consequently his financial health depends on that of his obligors. 

The strategic risk sharing behavior of economic actors thus induces a widely ramified network with important consequences in terms of financial stability. The highly desirable benefits from risk diversification are however balanced with the potential negative effects in terms of financial instability due to the network complexity. Agents holding claims on their 
debtors may experience a shock, and become unable to face their promised liability, thus impacting other actors through these linkages, and thereby creating a domino effect of financial distress. This contagion problem is at the heart of the so-called systemic risk, as it may possibly lead to a systemic crisis illustrated by the deficiency of a large part of the network. 

The analysis of economic actors interconnexions and their consequences in 
terms of financial stability has gained considerable attention in the literature on financial intermediaries, see e.g. \citeauthor*{allen2000financial} \cite{allen2000financial}, \citeauthor*{giesecke2004cyclical} \cite{giesecke2004cyclical}, and \citeauthor*{shin2009securitisation} \cite{shin2009securitisation}, \citeauthor*{eisenberg2001systemic} \cite{eisenberg2001systemic}, and \citeauthor*{acemoglu2015systemic} \cite{acemoglu2015systemic}.

In addition, as a major consequence of the last financial crisis, the financial regulation is relying on system-wide stress tests in order to evaluate the resilience of the financial system and the solidity of some key financial actors in the network. The technical development underlying these tools are highly difficult to conduct, due to the complexity of the network and the corresponding huge data bases. We refer to \citeauthor*{aikman2019system} \cite{aikman2019system}. In particular, this paper suggests certain dynamics of the market actors strategies in reaction to some market shocks. Such dynamics are set as rules for the simulation of the network system evolution, and thus defines the so-called agent-based model assumed to generate the equilibrium dynamics of the financial system. These rules are certainly reasonable, but have no reason to be optimal in any 
sense. Moreover, there is no justification that the limiting situation achieved after some steps of the dynamic simulation of the agent-based model is not known  to converge to some equilibrium in some sense.

We also refer to the continuous time literature which emerged during the last decade motivated by the development of McKean-Vlasov stochastic differential equations in financial modeling. \citeauthor*{garnier2013large} \cite{garnier2013large} analyze the large deviations of an interacting system of SDEs with a bistable potential as stabilizing force. \citeauthor*{carmona2013mean} \cite{carmona2013mean} model the network by a bird flocking mean field game model. \citeauthor*{hambly2019spde} \cite{hambly2019spde} analyze the scaling limit of an interacting system of defaultable agents with prescribed dynamics, \citeauthor*{nadtochiy2019particle} \cite{nadtochiy2019particle} prove existence for a mean field SDE with a singular dependence on the law of the hitting time of the origin, see also \citeauthor*{bayraktar2020mckean} \cite{bayraktar2020mckean}.

In this paper, we focus on the agents strategic behavior of building their interconnexions due to cross--holding. Each actor has some prescribed idiosyncratic risk, and seeks for a risk diversification of her position by 
optimally holding a proportion of the other actors. As all actors are simultaneously seeking this risk diversification objective, this leads to a 
situation of risk diversification of interacting agents, that we address by analyzing the existence of a Nash equilibrium of this game of mutual holding. 

In this paper, we focus on the mean field limiting model of a representative agent optimal mutual holding problem. The mean field limit is formally justified 
by symmetry arguments motivated by the anonymity of the system actors, so 
that their actions are represented through their distribution. In this paper, we restrict our attention to the context of independent idiosyncratic risk, and we leave the important case with common noise to future work. 
The mean field dynamics of the representative actor's equity value is new 
to the 
literature on the McKean-Vlasov type of mean field stochastic differential equations, as it exhibits two novel features. First, the corresponding coefficients may be discontinuous with respect to the state. Second, the dynamics depends on the law of the process through the expectation of a stochastic integration with respect to a copy of the process. The last feature is simplified at the equilibrium solution of our problem, as it reduces to an integral with respect to the equilibrium drift. Consequently, our proof of existence mainly addresses the discontinuity of the diffusion coefficient in the uniformly elliptic setting.

Our main results provide first an explicit Nash equilibrium of our mean field game of mutual holding. Due to the particular feature of our mean field SDE, the corresponding notion of mean field game is also outside the scope of the mean field games introduced by \citeauthor*{lasry2007mean} \cite{lasry2007mean}, and \citeauthor*{huang2003individual} \cite{huang2003individual}. From the PDE point of view, it may be expressed as a coupled system of backward HJB equation and a forward Focker-Planck equation combined 
with an integral equation for the equilibrium drift coefficient of the state. The latter is a novel feature for the present mean field game of mutual holding.

Our explicit mean field equilibrium is of bang--bang type, and  corresponds to the situation where the representative agent fully holds at each time $t$ those competitors whose drift coefficient is above some threshold $c(t,\mu_t)$ depending on the distribution $\mu_t$ of the population, and has 
no connexion with those competitors whose drift coefficient is below this 
threshold.

{\color{black}
Our next main result provides an approximate Nash equilibrium for the associated $N$--player game construct via the mean field game equilibrium. Due to the specificity of our mean field game problem, this construction is unusual in the MFG literature and to achieve it, we have to face technical difficulties related to the dynamics of the players.
}

The paper is organized as follows. Section \ref{sec:mutualholding} contains the formulation of the mutual holding problem, and introduces the corresponding mean field dynamics. Section \ref{sec:mfg} introduces our mean field game of mutual holding. The main results are 
contained in Section \ref{sec:mainresults}. Section \ref{sec:intuition} translates our mean field game formulation in the PDE language, thus highlighting the difference between our problem and the standard mean field games literature. {\color{black} Section \ref{sec:numerical} provides some numerical illustration of our main results in a toy example.} The proof of our characterization of the optimal mutual holding problem is reported in Section \ref{sec:proofmfg}. {\color{black}In Section \ref{sec:approx-nash}, we present the proof of the construction of approximate Nash equilibria for the $N$--player game. Finally Sections \ref{sec:SDE} and \ref{sec:density} are dedicated to some general technical results: existence of solution for a class of mean field stochastic differential equations with irregular coefficients and existence of density for the limiting distribution of a uniformly elliptic particles system.} 

\section{Interaction by mutual holding} \label{section_1}
\label{sec:mutualholding}

This section introduces the finite agent mutual holding model, and introduces the main mean field mutual holding model of this paper.\footnote{The model formulation results from various discussions of the second author with Charles Bertucci during most interesting interactions with the financial stability services of the Banque de France.} Our approach is to solve this mean field game problem, and to deduce an approximate equilibrium for the finite population Nash equilibrium. The convergence of the finite population problem to the mean field one follows as a by-product of this approach.

We consider $N$ economic agents with equity value induced by the idiosyncratic risk processes defined by the dynamics 
 \begin{align*}
    P^i_t
    =
    P^i_0+\int_0^t b^i_s \mathrm{d}s + \int_0^t \sigma^{i}_s \mathrm{d}W^i_s 
    + \int_0^t \sigma^{i,0}_s \mathrm{d}W^0_s,\;t \in [0,T],
    \;i=1,\ldots,N,
 \end{align*}
where $T>0$ is a fixed finite maturity, $W^0,\ldots,W^N$ are independent Brownian motions on a filtered probability space $(\Omega,\Fc,\F=\{\Fc_t\}_{0 \le t\le T},\P)$, and $b^i,\sigma^i,\sigma^{i,0}$ are coefficients
satisfying appropriate conditions for the wellposedness of such dynamics. Here $W^0$ is a common noise which affects the dynamics of all the $P_i$'s. Let us observe that the analysis of the present paper will be conducted in the case $\sigma^{i,0}=0$ for all $i=1,\ldots,N$. However, we shall provide the model formulation under the general situation so as to emphasize the main difficulties raised by the presence of common noise, and we leave the analysis of this setting to future research.

We assume that agents are allowed to hold each other as follows. At each time $t$, Agent $i$ chooses her optimal holdings $\pi^{i,j}_t$ in each of 
her competitors $j=1,\ldots,N$, $j\neq i$. As usual, we may reduce to the zero interest rates setting, without loss of generality, by expressing 
all amounts in terms of discounted values. Then, under the self-financing 
condition, the dynamics of the value process $X^i$ of the equity value of 
agent $i$ is given by:
 \begin{align} \label{eq_nagents-intuition}
 \mathrm{d}X^i_t
 =
 \mathrm{d}P^i_t
 + \sum_{j=1}^N \pi^{i,j}_t \mathrm{d}X^j_t
 - \sum_{j=1}^N \pi^{j,i}_t \mathrm{d}X^i_t,
 ~~i=1,\ldots,N.
 \end{align}
The first summation in the last dynamics indicates agent $i$'s returns from the holdings in the competitors, while the second one records the returns of the other competitors from holding agent $i$.

Given the strategies $\big(\pi^{i,\cdot}:=(\pi^{i,1},\cdots,\pi^{i,N}) \big)_{1 \le i \le N},$ the reward of Agent $i$ is given by 
\begin{align*}
    J_i\big( \pi^{1,\cdot},\cdots,\pi^{N,\cdot} \big)
    :=
    \E^{\P} \big[ U\big( X^i_T \big) \big].
\end{align*}
Each Agent $i$ aims at maximizing her reward following a Nash equilibrium criterion. More precisely, $\big( \pi^{1,\cdot},\cdots,\pi^{N,\cdot} \big)$ is a Nash equilibrium if none of the players can gain by deviating from it, i.e. for any admissible strategy $\beta:=(\beta^{1},\cdots,\beta^N)$, 
\begin{eqnarray}\label{N-natural}
    J_i\big( \pi^{1,\cdot},\cdots,\pi^{N,\cdot} \big)
    \ge 
    J_i\big( \pi^{1,\cdot},\cdots,\pi^{i-1,\cdot},\beta,\pi^{i+1,\cdot},\cdots,\pi^{N,\cdot} \big),
    ~\mbox{for all}~i=1,\ldots,N.
\end{eqnarray}

Our objective is to analyze the optimal strategic holdings $(\pi^{i,\cdot})_{1 \le i \le N}=(\pi^{i,j})_{1 \le i,j \le N}$ in the present context where each agent controls her own holdings in the others, and undergoes her competitors' holding decision, in particular in her own asset.

This paper first focuses on the mean field limit corresponding to an appropriate large population scaling limit induced by the following symmetry considerations. In order to formulate such conditions, we introduce the agent's values empirical measures 
$$
\mu^N_t:=\frac1N\sum_{i=1}^N\delta_{X^i_t},~~t \in [0,T].
$$ 
We first impose the anonymity of the idiosyncratic value process by restricting the 
coefficients defining its dynamics  by 
$$
b^i_t=b(t,X^i_t,\mu^N_t),~~
\sigma^i_t=\sigma(t,X^i_t,\mu^N_t),~~
\mbox{and}~~\sigma^{i,0}_t=\sigma^0(t,X^i_t,\mu^N_t).
$$ 
The second key assumption guarantees the behavioral anonymity of agents by restricting their holdings decisions to depend only on their own equity value, that of the investment competitor, and possibly the empirical measure $\mu^N_t$. Then assuming that the proportions of holdings are of order $\frac1N$, we may consider the last dynamics as:
 \b*
 \mathrm{d}X^i_t
 &=&
 \frac1N\sum_{j=1}^N \beta(t,X^i_t,X^j_t,\mu^N_t) \mathrm{d}X^j_t
 - \frac1N\sum_{j=1}^N \pi(t,X^j_t,X^i_t,\mu^N_t) \mathrm{d}X^i_t \nonumber
 \\
 &&
 +b(t,X^i_t,\mu^N_t)\mathrm{d}t+\sigma(t,X^i_t,\mu^N_t)\mathrm{d}W^i_t+\sigma^0(t,X^i_t,\mu^N_t)\mathrm{d}W^0_t,
 \e*
where we abuse notations by setting $\pi^{j,i}_t=\frac1N\pi(t,X^j_t,X^i_t,\mu^N_t)$, and we denoted agent $i$'s decision variable differently by $\beta$ because it will be the main focus of the individual optimization component of our equilibrium definition. 

By the independence of the Brownian motions $(W^i)_{i\ge 1}$, we now guess that the infinite number of agents limit $N\to\infty$ exhibits a propagation of chaos effect, conditional on the common noise $W^0$. According to the classical mean field game intuition, the contribution of player $i$ 
in the empirical distribution $\mu^N_t$ is negligible. This leads naturally to postulate the following mean field game problem: a $W^0-$measurable 
probability measure $\mu$ is MFG equilibrium if 
\begin{enumerate}
\item[{\rm (i)}] $\mu=\P^{W^0} \circ (X^{\pi,\pi})^{-1}$ is the law of $X^{\pi,\pi}$ conditional on the common noise $W^0$, $\mu_s=\P^{W^0}\circ (X^{\pi,\pi}_s)^{-1}$, where $X^{\beta,\pi}$ is governed by the dynamics
 \b*
 X_t
 &\!\!\!=&\!\!\!
 X_0
 +\widehat{\E}^{\mu}\bigg[ \int_0^t \!\!\!\beta(s,X_s,\widehat{X}_s, \mu_s) \mathrm{d}\widehat{X}_s \bigg]
 \;-\;\int_0^t \!\widehat{\E}^{\mu}\!\big[\pi(s,\widehat{X}_s,X_s, \mu_s)\big] \mathrm{d}X_s
 \\
 &\!\!\!\!\!&\!\!\!
 +\;\int_0^t b(s,X_s,\mu_s)\mathrm{d}s\;+\;\int_0^t\sigma(s,X_s,\mu_s)\mathrm{d}W_s
 \;
 +\;\int_0^t\sigma^0(s,X_s,\mu_s)\mathrm{d}W^0_s,
 \e*
 with  $\{\widehat{X}_t,t\ge 0\}$ denoting the canonical process on the space of continuous paths, and the operator $\widehat\E^\mu[\cdot]:=\int \cdot\, \mu(d\hat x)$ the corresponding ($W^0-$conditional) expectation operator under $\mu$,
\item[{\rm (ii)}] and the following optimality condition holds for the 
control $\pi$
 \b*
     \E^{\P} \big[ U \big( X^{\pi,\pi}_T \big) \big]
     &\ge& 
     \E^{\P} \big[ U \big( X^{\beta,\pi}_T \big) \big],\;\;\mbox{for all}\;\beta.
 \e*
\end{enumerate}
Due to the integration with respect to the copy $\widehat{X}$, the last mean field dynamics lies outside the scope of standard McKean-Vlasov type of mean field stochastic differential equations. 

Moreover, the optimization problem is not standard in the MFG literature, 
in particular because in addition to the measure $\mu,$ a part of the control $\pi$ is also fixed. Our objective is to provide a precise meaning to it, and to define a notion of mean field game of mutual holding by an appropriate Nash equilibrium within the population of mean field interacting agents.

In the situation with no common noise $\sigma^0\equiv 0$, we rewrite the last dynamics in the following simpler form:
 \be\label{controlled}
 X_t
 &\!\!\!=&\!\!\!
 X_0
 +\widehat{\E}^{\mu}\bigg[ \int_0^t \!\!\!\beta(s,X_s,\widehat{X}_s) \mathrm{d}\widehat{X}_s \bigg]
 - \int_0^t \!\widehat{\E}^{\mu}\!\big[\pi(s,\widehat{X}_s,X_s)\big] \mathrm{d}X_s
 \\
 &\!\!\!&\!\!\!
 +\int_0^t b\big(s,X_s,\mu_s \big)\mathrm{d}s+ \int_0^t\sigma\big(s,X_s,\mu_s \big)\mathrm{d}W_s,
 \nonumber
 \ee
 where the dependence of the decision variables $\beta$ and $\pi$ on the law is absorbed by the their dependence on the time variable, and $\widehat\E^\mu$ is the expectation operator on the canonical space of continuous paths. Here again the mean field dependence is new to the literature as it involves the mean of a stochastic integral with respect to a copy of the solution.
 
\section{Mean field game of mutual holding}
\label{sec:mfg}

\subsection{Mutual holding mean field SDE}

Let $T>0$ be a finite maturity. We denote by $\widehat{X}$ the canonical process on the paths space $\widehat\Omega:=C^0([0,T],\R)$, i.e. $\widehat X_t(\mathbf{x})=\mathbf{x}(t)$ for all $t \in [0,T],$ and $\mathbf{x}\in\widehat\Omega$. The corresponding raw filtration $\widehat\F^0:=\big\{\widehat\Fc^0_t,t\in[0,T]\big\}$ is defined by $\widehat\Fc^0_t:=\sigma\{\widehat X_s,s\le t\}$, and we shall work throughout with the right limit of its universal completion $\widehat\F:=\big\{\widehat\Fc_t,t\in[0,T]\big\}$, defined by $\widehat\Fc_t:=\lim_{s\searrow t}\Fc^{\rm U}_s$, with $\Fc^{\rm U}_t:=\cap_{\P\in{\rm Prob}(\widehat\Omega)}(\widehat\Fc^0_t)^\P$, and ${\rm Prob}(\widehat\Omega)$ the collection of all probability measures on $\widehat\Omega$.

We also fix some initial distribution $\nu \in \Pc_{p}(\R)$, with $p>2.$ Let $\Pc_\Sc$ be the collection of all probability measures $\mu$ on $\widehat\Omega$ such that $\widehat X$ is a $\mu-$square integrable It\^o process, i.e. 
 \be\label{BmuSigmamu}
 \widehat X_t &:=& \widehat X_0+\int_0^t B^\mu_r\, \mathrm{d}r+\int_0^t 
\Sigma^\mu_r \,\mathrm{d}W^\mu_r,~~t\in [0,T],~~
 \mu\mbox{--a.s. with}~~
 \mu\circ \widehat X_0^{-1}=\nu, 
 \ee
 for some $\mu$--Brownian motion $W^\mu$, and some $\widehat\F-$progressively measurable deterministic functions $B^\mu,\Sigma^\mu:[0,T]\times\widehat\Omega\longrightarrow\R$, with finite $\H^2_\mu$ norm $\|B^\mu\|^2_{\H^2_\mu}:=\widehat\E^\mu\big[\int_0^T |B^\mu_t|^2\mathrm{d}t\big]<\infty$ and $\|\Sigma^\mu\|_{\H^2_\mu}<\infty$. Here, $\widehat\E^\mu$ denotes 
the corresponding expectation operator on $(\widehat\Omega,\widehat\F,\mu)$. 
 
We now return to the mutual holding problem of the previous section with mutual holding strategy defined by a measurable function 
 $$
 \pi:[0,T]\times\R\times\R\longrightarrow[0,1].
 $$ 
We denote by $\Ac$ the collection of all such maps. Motivated by the discussion of the previous section, we fix some distribution $\mu\in\Pc_\Sc$, 
and we consider for all such holding strategy $\pi\in\Ac$ the following mean field SDE driven by a $\P-$Brownian motion $W$ on a filtered complete 
probability space $(\Omega,\Fc,\F=\{\Fc_t\}_{t \in [0,T]},\P)$:
 \b*
 X_t
 &\!\!=&\!\!
 X_0
 +\widehat\E^\mu\bigg[ \int_0^t \pi(s,X_s,\widehat X_s) \mathrm{d}\widehat X_s \Big]
 - \int_0^t \widehat\E^\mu\big[\pi(s,\widehat X_s,X_s)\big] \mathrm{d}X_s
 \\
 &&
 +\int_0^t b(s,X_s,\mu_s)\mathrm{d}s
 +\int_0^t \!\!\sigma(s,X_s,\mu_s)\mathrm{d}W_s
 \\
 &\!\!\!\!=&\!\!\!\!
 X_0
 +\!\!\int_0^t \!\!\Big(b(s,X_s,\mu_s)+\widehat\E^\mu\big[ \pi(s,X_s,\widehat X_s) B^\mu(s,\widehat X_{s\wedge .})\big]\Big)\mathrm{d}s
 -\!\! \int_0^t \!\!\widehat\E^\mu\big[\pi(s,\widehat X_s,X_s)\big] \mathrm{d}X_s~~~~
 \\
 &&
 \hspace{33mm}+\int_0^t \sigma(s,X_s,\mu_s)\mathrm{d}W_s,
 \e*
or in differential form:
 \begin{equation} \label{MF-SDE}
 \mathrm{d}X_s
 =
 \frac{b(s,X_s,\mu_s)+\widehat\E^\mu\big[ \pi(s,X_s,\widehat X_s) B^\mu(s,\widehat X_{s\wedge .})\big]}
        {1+\widehat\E^\mu\big[\pi(s,\widehat X_s,X_s)\big]}\mathrm{d}s
 +\frac{\sigma(s,X_s,\mu_s)}
        {1+\widehat\E^\mu\big[\pi(s,\widehat X_s,X_s)\big]}\mathrm{d}W_s.
 \end{equation}
Let $\Pc_\Sc(\pi)$ be the (possibly empty!) subset  consisting of all measures $\mu\in\Pc_\Sc$ such that the last SDE
has a weak solution $X\in\Sc$ satisfying $\mu=\P\circ X^{-1}$. 

Comparing \eqref{MF-SDE} with \eqref{BmuSigmamu}, we see that for $\mu\in\Pc_\Sc(\pi)$, 
 \b*
 B^\mu(t,\mathbf{x}(t\wedge \cdot))=B^\mu(t,\mathbf{x}(t))
 ,~~
 \Sigma^\mu(t,\mathbf{x}(t\wedge \cdot))=\Sigma^\mu(t,\mathbf{x}(t)),
 &t\in [0,T],&
 \mathbf{x}\in\widehat\Omega,
 \e* 
depend only on the current value of the path, and we obtain the following 
identification of the drift and diffusion coefficients of the mean field SDE:
 \begin{equation}\label{identification}
 B^\mu(t,x)=\frac{b(t,x,\mu_t)+\int_{\R} \pi(t,x,y)B^\mu(t,y)\mu_t(\mathrm{d}y)}
                           {1+\int_{\R} \pi(t,y,x)\mu_t(\mathrm{d}y)},
 ~~
 \Sigma^\mu(t,x)=\frac{\sigma(t,x,\mu_t)}
                           {1+\int_{\R} \pi(t,y,x)\mu_t(\mathrm{d}y)},
 \end{equation}
 $ \mathrm{d}t\otimes\mu_t(\mathrm{d}x)$--a.e. 
 
 \subsection{Mean field game formulation}
 \label{sect:MFGformulation}
 
 For all $\pi\in\Ac$ and $\mu\in\Pc_\Sc(\pi)$, the deviation of the representative agent from the mutual holding strategy $\pi\in\Ac$ to an alternative one $\beta\in\Ac$ is defined by introducing an equivalent probability measure $\P^\beta_{\pi,\mu}$ via the density with respect to $\P$: 
 $$
 \frac{\mathrm{d}\P^\beta_{\pi,\mu}}{\mathrm{d}\P}
 :=
 e^{\int_0^T \psi_t\mathrm{d}W_t-\frac12| \psi_t|^2\mathrm{d}t},
 ~\mbox{on}~\Fc_T,
 $$
where the process $\psi=\psi^\beta_{\pi,\mu}$ is defined by
 \be\label{psi}
 \psi_t
 :=
 \frac{\widehat\E^\mu\big[(\beta-\pi)(t,X_t,\widehat X_t)B^\mu(t,\widehat 
X_t)\big]}{\sigma(t,X_t,\mu_t)}
 =
 \frac{\int_{\R} (\beta-\pi)(t,X_t,y)B^\mu(t,y)\mu_t(\mathrm{d}y)}{\sigma(t,X_t,\mu_t)}.
 \ee
We shall assume below that the diffusion coefficient is bounded away from 
zero, so that the above change of measure is well-defined by the boundedness of $\pi$ and $\beta$ and the $\mu$--square integrability of $B^\mu$. Moreover, it follows from the Girsanov Theorem that the process $W^{{\pi,\mu,\beta}}_\cdot:=W_\cdot-\int_0^\cdot \psi_s\mathrm{d}s$ is a $\P^\beta_{\pi,\mu}-$Brownian motion, so that the $\P^\beta_{\pi,\mu}-$dynamics of the value process $X$ are given by
 \b*
 X_\cdot
 &=&
 X_0
 +\widehat\E^\mu\bigg[ \int_0^\cdot \beta(s,X_s,\widehat X_s) \mathrm{d}\widehat X_s \bigg]
 - \int_0^\cdot \widehat\E^\mu\big[\pi(s,\widehat X_s,X_s)\big] \mathrm{d}X_s
 \\
 &&
 \hspace{7mm}
 +\int_0^\cdot b(s,X_s,\mu_s)\mathrm{d}s
 +\int_0^\cdot \sigma(s,X_s,\mu_s)\mathrm{d}W^{\pi,\mu,\beta}_s,
 \\
 &=&
 X_0
 +\int_0^\cdot \Big( b(s,X_s,\mu_s)+\widehat\E^\mu\big[\beta(s,X_s,\widehat X_s) B^\mu(s,\widehat X_s)\big] \Big)\mathrm{d}s
 - \int_0^\cdot \widehat\E^\mu\big[\pi(s,\widehat X_s,X_s)\big] \mathrm{d}X_s
 \\
 &&
 \hspace{40mm}+\int_0^\cdot \sigma(s,X_s,\mu_s)\mathrm{d}W^{\pi,\mu,\beta}_s,
 \e*
thus mimicking the controlled dynamics in \eqref{controlled} when there is no common noise $\sigma^0\equiv 0$.

The representative agent seeks for an optimal mutual holding strategy by maximizing her criterion
 \b*
 J_{\pi,\mu}(\beta)
 :=
 \E^{\P^\beta_{\pi,\mu}}\big[U(X_T)\big]
 &\mbox{over all}&
 \beta\in\Ac,
 \e*
where $U:\R\longrightarrow\R$ is a given non-decreasing utility function.

\begin{Definition}\label{def:MFG}
 A pair $(\pi,\mu)\in\Ac\times\Pc_\Sc$ is a mean field game equilibrium of the mutual holding problem if
\\
{\rm (i)} $\mu\in\Pc_\Sc(\pi)$, i.e. the mean field SDE \eqref{MF-SDE} has a solution with law $\mu$.
\\
{\rm (ii)} $J_{\pi,\mu}(\pi)\ge J_{\pi,\mu}(\beta)$ for all $\beta\in\Ac$, i.e. $\pi$ is an optimal response for the representative agent mutual holding problem.
\end{Definition}

\medskip

\begin{Remark} \label{rm:definition}
{\rm
(i) In our definition $\pi$ represents the optimal control, and $\mu$ the 
equilibrium distribution of the optimal process controlled by $\pi.$ Due to the potential singularity of the coefficients of the SDE in our problem, see the equilibrium dynamics in  Theorem \ref{thm:main}, note that uniqueness is in general not guaranteed and so $\pi$ does not determine $\mu$. Conversely, the knowledge of $\mu$ does not determine $\pi.$ This is why our definition of equilibrium involves the pair $(\pi,\mu)$.

\noindent (ii) Condition {\rm (i)} of Definition \ref{def:MFG} is the analogue of the consistency condition used in the MFG literature, see \citeauthor*{carmona2018probabilisticI} \cite{carmona2018probabilisticI}, while 
Condition {\rm (ii)} is the usual representative agent optimality condition. 

\noindent (iii) Notice that  the representative agent control $\beta$ acts only on the drift coefficient, and that the volatility is not controlled. This is an unusual feature in comparison with standard portfolio optimization problems. Remarkably, the volatility control was present in the $N-$agents microscopic formulation of Section \ref{sec:mutualholding}, and 
has disappeared in the mean field formulation due to the absence of common noise. However, even though the representative agent only controls the drift, notice that the optimal control impacts the volatility through the 
equilibrium mutual holding $\pi$. The case with common noise generates many additional difficulties such as controlling volatility, a situation that we leave for future research. 
}
\end{Remark}

\section{Main results}
\label{sec:mainresults}

This section provides the characterization of a solution of the mean field game of mutual holding. We first state our conditions on the coefficients of the SDE defining the dynamics of the idiosyncratic risk process
 \b*
 b,\sigma:~[0,T]\times\R\times\Pc_2(\R)\longrightarrow \R,
 \e*
where $\Pc_2(\R)$ is the collection of all square integrable laws on $\R$, endowed with the 2-Wasserstein distance
 \b*
 \Wc_2(m,m')^2
 :=
 \inf_{\gamma\in\Pi(m,m')}\int_{\R^2} |x-x'|^2\gamma(\mathrm{d}x,\mathrm{d}x'),
 &m,m'\in\Pc_2(\R),&
 \e*
with $\Pi(m,m')$ the collection of all probability measures on $\R^2$ with marginals $m$  and $m'$.

\begin{Assumption}\label{assum:bsigma}
The coefficients $b$ and $\sigma$ are Borel measurable, Lipschitz in $m,$ 
with quadratic growth in $(x,m)$, uniformly in $t\in[0,T]$, and $\sigma$ is bounded from below away from zero. Moreover,
\vspace{-4mm}\begin{itemize}

\item[\rm (i)] either for each $m,$ $b(t,x,m) < 0,$ for a.e. $(t,x) \in [0,T] \x\R$,

{\color{black}\item[\rm (ii)]or for any $\eta \in \R_+$ and $m \in \Pc(\R),$ for Lebesgue--a.e. $t \in [0,T],$ the Borel set
\begin{align*}
    \ell (t,m,\eta)
    :=
    \big\{
        x \in \R:(x',m') \longmapsto\mathbf{1}_{\{b(t,x',m') + \eta \ge 0 \}}\mbox{ is continuous at the point } (x,m)
    \big\}
\end{align*}
has full Lebesgue measure i.e. its complement is Lebesgue--negligible .
}
\end{itemize}

\end{Assumption}
{\color{black}
\begin{Remark}
{\rm
Condition ${\rm (ii)}$ holds true when the drift of the provisions process $b$ is non--negative. It also holds true when the map $x \longmapsto b(t,x, m)$ is continuous and has Lebesgue--negligible negative level sets, for all $(t,m) \in [0,T] \x \Pc(\R)$. This includes the case of affine drift as for the Ornstein–Uhlenbeck process.
}
\end{Remark}
}
{\color{black}
\subsection{An explicit solution of the MFG of mutual holding}
}
\begin{Theorem}\label{thm:main}
Let Assumption \ref{assum:bsigma} hold true, and let the agent's criterion $U$ be non-decreasing Lipschitz. Then, there exists a solution $(\pi^\star,\mu)$ for the mean field game of mutual holding, with 
\begin{itemize}
\item[{\rm (i)}] equilibrium dynamics of the state defined by $\big(B^\mu,\Sigma^\mu\big)(t,x)=(B,\Sigma)(t,x,\mu_t)$:
\b*
B
:=
\frac12(b+c)^+-(b+c)^-, 
&\mbox{and}&
\Sigma
:=
\Big(1-\frac12\1_{\{b+c\ge 0\}}\Big)\sigma,
\e*
where $c(t,m)\ge 0$ is the unique solution of the equation 
\begin{align} \label{eq:charac_c}
    c=\frac12\int_\R \big(c+b(t,y,m)\big)^+m(\mathrm{d}y).
\end{align}
\item[{\rm (ii)}] optimal mutual holding map $\pi^*(t,x,y)
 =
 \1_{\{B^\mu(t,y)\ge 0\}},$ $t\in[0,T],$ $x,y\in\R$,
\item[{\rm (iii)}] and equilibrium value given by $J_{\pi^*,\mu}(\pi^*)=\E^{\P}\big[U(X_T)\big]=\widehat\E^{\mu}\big[U(\widehat X_T)\big]$.
\end{itemize}
\end{Theorem}

\begin{Remark} {\rm
{\rm  (i)} At first sight, the reader may think that $\pi^\star$ is entirely determined by $\mu.$ But, as mentioned in Remark \ref{rm:definition}, this 
is not the case. To see this, notice that the expression \eqref{identification} of the drift at the equilibrium $B$ involves the optimal control $\pi^\star$. Therefore, through its dependence on $B$, $\mu$ and $\pi^\star$ are closely related, consequently knowing one does not lead to knowing the 
other.

 {(ii)} The optimal control in Theorem \ref{thm:main} (ii) is bang-bang, and consists in investing only when $b(t,X_t,\mu_t)$ is above the threshold $-c(t,\mu_t)$, where $X$ represents the equilibrium state of the mean field game of mutual holding, and $c$ solves the equation mentioned in the theorem. Notice that this equilibrium behavior reduces the volatility by half on the region where the drift is positive. This may be interpreted 
as a stabilization effect due to risk diversification.
        
}
\end{Remark}

The proof of items (ii) and (iii) of this result is reported in the following section. For the sake of clarity, we justify at the end of the present section the existence of the function $c(t,m)$ satisfying \eqref{eq:charac_c}. Before this, we provide the justification of the equilibrium dynamics in (i) as a consequence of the optimal mutual holding strategy $\pi^*$. Plugging the expression of $\pi^*$ in \eqref{identification}, we see that
 \be\label{Bmu*}
 B^\mu(t,x)
 &=&
 \frac{b(t,x,\mu_t)+\int_{\R} B^\mu(t,y)^+\mu_t(\mathrm{d}y)}
        {1+\1_{\{B^\mu(t,x)\ge 0\}}}
 \ee
Multiplying both sides by $\1_{\{B^\mu(t,x)\ge 0\}}$, yields
 $$
 2B^\mu(t,x)\1_{\{B^\mu(t,x)\ge 0\}}
 =
 \1_{\{B^\mu(t,x)\ge 0\}}\Big[b(t,x,\mu_t)+\int_{\R} B^\mu(t,y)^+\mu_t(\mathrm{d}y)\Big],
 $$
which implies by integration with respect to $\mu_t$ and the fact that $\{x:\; (b+c)(t,x,\mu_t) \ge 0 \}=\{x:\; B^\mu(t,x) \ge 0 \},$ 
\b*
\int\!\! \1_{\{B^\mu(t,y)\ge 0\}}B^\mu(t,y)\mu_t(\mathrm{d}y)
&\!\!\!\!\!\!=&\!\!\!\!\!\!
\frac{\int_{\R} \!\1_{\{B^\mu(t,y)\ge 0\}}b(t,y,\mu_t)\mu_t(\mathrm{d}y)}{2-\int \!\1_{\{B^\mu(t,y)\ge 0\}}\mu_t(\mathrm{d}y)}
\\
&\!\!\!\!\!\!=&\!\!\!\!\!\!
\frac{\int (b\!+\!c)^+\!(t,y,\mu_t)\mu_t(\mathrm{d}y) \!-\! c(t,\mu_t)\!\int\! \1_{\{(b+c)(t,y,\mu_t)\ge 0\}}\mu_t(\mathrm{d}y)}{2-\int \!\1_{\{(b+c)(t,y,\mu_t)\ge 0\}}\mu_t(\mathrm{d}y)}
\\
&=&
\frac{2c(t,\mu_t) \!-\! c(t,\mu_t)\!\int\! \1_{\{(b+c)(t,y,\mu_t)\ge 0\}}\mu_t(\mathrm{d}y)}{2-\int \!\1_{\{(b+c)(t,y,\mu_t)\ge 0\}}\mu_t(\mathrm{d}y)}
\;=\;
c(t,\mu_t),
\e*
by the definition of $c(t,\mu_t)$. Substituting in \eqref{Bmu*}, we get $B^\mu(t,x)=(1+\1_{\{B^\mu\ge 0\}})^{-1}(b+c)(t,x,\mu_t)$, and this leads to the MFG equilibrium mean field SDE
 \be\label{MFG-SDE} 
 \mathrm{d}X_t
 &=&
 B(t,X_t,\mu_t)\mathrm{d}t+\Sigma(t,X_t,\mu_t)\mathrm{d}W_t,
 \ee
with $B$ and $\Sigma$ as in Theorem \ref{thm:main} (i).

Notice that the drift coefficient of the last SDE is continuous across the boundary $\{B=0\}$, while the diffusion coefficient is not continuous 
across this boundary. For this reason, we shall introduce a notion of weak solution for a general class of McKean-Vlasov SDEs which covers our setting. The following result states an existence result for the last SDE which holds despite the discontinuity of the diffusion coefficient.

\begin{Theorem}\label{thm:SDE}
Under Assumption \ref{assum:bsigma}, the mean field SDE \eqref{MFG-SDE} has at least one square integrable weak solution in the sense of Definition \ref{def:SDE}.
\end{Theorem}

This result is a direct consequence of the more general wellposedness result, Theorem \ref{thm:weak_existence}, reported in Section \ref{sec:SDE} below. {\color{black}Given this result, we now report the following consequence of the verification argument in Proposition \ref{prop:optimization}, which is one of the main steps of the proof of Theorem \ref{thm:main}.

\begin{Corollary} \label{cor:MFG_sol}
Any weak solution $\mu$ of \eqref{MFG-SDE} is a solution of the mean field game of mutual holding with optimal control $\pi^\star(t,x,y):=\mathbf{1}_{\{B(t,y,\mu_t) \ge 0\}}$, for all $(t,x,y)\in[0,T] \x \R \x \R$.
\end{Corollary}
}

\begin{Remark}
{\rm Let us examine the case when $b$ has a constant sign. 
\begin{itemize}
\item Let $b(\cdot,m) \ge 0,$ a.e. on $[0,T] \x \R$, for all $m$. then we compute that $c(t,m)=\int_\R b(t,y,m) m(\mathrm{d}y),$ $B=\frac{1}{2}(b+c)$, $\Sigma=\frac12\sigma$, and Theorem \ref{thm:weak_existence} applies.
\item Let $b(\cdot,m)< 0,$ a.e. on $[0,T] \x \R$,  for all $m$. Then, we immediately see that $c(t,m)=0$, so that $B=b$, and $\Sigma=\big(1-\frac12\1_{\{b\ge 0\}}\big)\sigma.$  Thanks to Theorem \ref{thm:weak_existence}, we define $\mu$ the distribution of the SDE solution associated to the coefficients $(b,\sigma),$ we 
can check that $\sigma=\big(1-\frac12\1_{\{b\ge 0\}}\big)\sigma,$ $\mu_t(\mathrm{d}x)\mathrm{d}t$--a.e. 
\end{itemize}}
\end{Remark}

We conclude this section with the justification of the existence and uniqueness of the  function $c:\R_+\times\Pc_2(\R)\longrightarrow\R$ in a general situation.

\begin{Lemma}\label{lem:cvarphi}
Let $\varphi:\R\times\Pc_2(\R)\longrightarrow\R$ be a function such that $\varphi(.,m)$ is $m$--integrable for all $m \in \Pc_2(\R).$ Then, there is a unique 
\b*
c^\varphi:\Pc_2(\R)\longrightarrow\R_+
&\mbox{such that}&
c^\varphi=\frac12\int_\R \big(c^\varphi+\varphi(x,m)\big)^+m(\mathrm{d}x).
\e* 
Moreover, $c^\varphi\le 2\int \varphi^+(x,m)m(\mathrm{d}x)$, and if in addition $\varphi$ is Lipschitz in $(x,m)$, then $c^\varphi$ is Lipschitz in $m$.
\end{Lemma}

\proof
Define $F(c):=c-\frac{1}{2} \int \big( c + \varphi(x,m) \big)^+ m(\mathrm{d}x),$ for $c\ge 0$. One has $F(0)=-\frac{1}{2} \int\big(\varphi(x,m) \big)^+ m(\mathrm{d}x) \le 0,$ and $\lim_{c \to \infty} F(c)=\infty.$ 

We can check that $F'(c)=1-\frac{1}{2} \int\mathbf{1}_{c \ge \varphi(x,m) }m(\mathrm{d}x),$ then $F'(c) \ge 1/2 >0.$ Consequently, there exists a unique $c(m) \in \R_+$ verifying the required equation $F(c)=0$. 

As, $c(m)
        =
        \frac{1}{2} \int \big(c(m)+\varphi)^+(x,m) m(\mathrm{d}x)
        \le \frac{1}{2} \int(c(m)+ \varphi^+)(x,m) m(\mathrm{d}x)$, it follows that $c(m) \le 2\int\varphi^+(y,m)  m(\mathrm{d}y)$.
 Following the same argument, we see that
     \b*
        |c(m)-c(m')| 
        &\le& 
        2\;\inf_{\gamma \in \Pi(m,m' )} \int_{\R^2} \big| \varphi^+(x,m) - \varphi^+(x',m') \big| \gamma(\mathrm{d}x,\mathrm{d}x').
    \e*
This implies that $c$ inherits from $\varphi$ the required Lipschitz property.
\ep

\subsection{Intuitions from PDE arguments}
\label{sec:intuition}

The mean field game problems are commonly formulated via a system of coupled PDEs. Because of the new form of this model, we provide here such a PDE formulation in order to provide more insights of our problem. 

The characterization through PDEs consists in a backward Hamilton--Jacobi--Bellman (HJB) equation coupled with a forward Fokker--Planck (FP) equation, see \citeauthor*{lasry2007mean} \cite{lasry2007mean}, \citeauthor*{huang2006large} \cite{huang2006large}. In our situation, in addition to the HJB and FP equations, we have an equation which identifies the drift coefficient at the equilibrium. Precisely, solving our MFG problem is reduced to finding the triple $(v,\mu,B)$ such that: $v:[0,T] \x \R \longrightarrow \R,$ $\mu: [0,T]\longrightarrow\Pc_2(\R),$ $B:[0,T] \x \R \x \Pc_2(\R) \longrightarrow\R$ satisfying 
\begin{itemize}
\item the HJB equation: 
\b*
   v(T,x)=U(x),~\mbox{and}~
   0
   &\!\!=&\!\!
   \partial_t v(t,x) + \frac{1}{2} \frac{\sigma^2(t,x,\mu_t) D^2 v(t,x)}
                                                                            {(1 + \int_{\R}\pi(t,y,x) \mu_t(\mathrm{d}y))^2}
   \\
   &\!\!&\!\!+ \frac{D v(t,x) b(t,x,\mu_t) +\int \sup_{\beta\in [0,1]} \beta\;B(t,y,\mu_t)Dv(t,x) \mu_t(\mathrm{d}y)}
               {1 + \int_{\R}\pi(t,y,x) \mu_t(\mathrm{d}y)},
  \\
     &\!\!=&\!\!
   \partial_t v(t,x) + \frac{1}{2} \frac{\sigma^2(t,x,\mu_t) D^2 v(t,x)}
                                                                            {(1 + \int_{\R}\pi(t,y,x) \mu_t(\mathrm{d}y))^2}
   \\
   &\!\!&\!\!+ \frac{D v(t,x) b(t,x,\mu_t) +\int \big(Dv(t,x)B(t,y,\mu_t) \big)^{+} \mu_t(\mathrm{d}y)}
               {1 + \int_{\R}\pi(t,y,x) \mu_t(\mathrm{d}y)}
               ,
\e*
where the supremun is attained at the point $\pi(t,x,y)=\mathbf{1}_{\{Dv(t,x)B(t,y,\mu_t)\;\ge\; 0\}}$, 
\item the FP equation:  
\begin{align*}
\mu_0=\nu,~\mbox{and}~
    \partial_t \mu_t
        &=
        -D \big[ B(t,\cdot,\mu_t) \mu_t \big] 
        +
        \frac{1}{2} D^2 \Big[ \frac{\sigma^2 (t,\cdot,\mu_t)\mu_t}
                                               {(1 + \int_{\R}\pi(t,y,x) \mu_t(\mathrm{d}y))^2}  \Big] ,
\end{align*}
\item and finally the equation identifying the drift at the equilibrium
\begin{align*} 
    B(t,x,\mu_t)=\frac{\int_{\R} \Big( \pi(t,x,y) B(t,y,\mu_t) + b(t,x,\mu_t) \Big) \;\mu_t(\mathrm{d}y)}
                                {1+\int_{\R}\pi(t,y,x) \mu_t(\mathrm{d}y)},
    ~~\mu_t(\mathrm{d}x)\otimes\mathrm{d}t-\mbox{a.e.}
\end{align*}
\end{itemize}
This system goes beyond the classical system appearing in the literature. 
Moreover, the non--local feature of this system makes its study difficult. Our main result of Theorem \ref{thm:main} offers a solution to this system of PDE in some appropriate weak sense.

{\color{black}
\subsection{From the solution of the MFG to \\ a finite population approximate Nash equilibrium}

Our last main result, Theorem \ref{thm:nash_equilibria} below, provides an explicit construction of (approximate) Nash equilibria for the $N$--player game. We start by a precise formulation of the $N$--player game version of our problem. 

Throughtout this section, $(\Om,\F,\Fc,\P)$ is a filtered probability space supporting a sequence of independent Brownian motions $(W^i)_{i \ge 1}$, and we denote $\mathbf{W}:=(W^1,\ldots,W^N)$.

\paragraph*{Semi--martingale representation} Let $\Gamma:=(\gamma^{i,j})_{1 \le i,j \le N}$ be a $[0,1]^{N \x N}$--valued $\F$--predictable process, and $\mathbf{X}:=(X^1,\cdots,X^N)$ the solution of the SDE
\begin{align*}
    \mathrm{d}X^i_t
    =
    \frac{1}{N}\! \sum_{j=1}^N\! \gamma^{i,j}_t \mathrm{d}X^j_t 
    \!-\!
    \frac{1}{N}\! \sum_{j=1}^N\! \gamma^{j,i}_t \mathrm{d}X^i_t
    \!+\!
    b(t,X^i_t,\mu^N_t) \mathrm{d}t
    \!+\!
    \sigma(t,X^i_t,\mu^N_t) \mathrm{d}W^i_t,
    ~\mu_t^N:=\frac1N\!\sum_{i=1}^N\!\delta_{X^i_t},
\end{align*}
with $\Lc(X^1_0,\cdots,X^N_0)=\mu_0^{\otimes N}$, for some $\mu_0 \in \Pc_p(\R)$, and $p >2.$ By rewriting this equation in vector notation, we see that $M(\Gamma_t)\mathrm{d}\mathbf{X}_t=\vec{b}(t,\mathbf{X}_t,\mu^N_t)\mathrm{d}t+\mbox{\rm diag}[\vec{\sigma}(t,\mathbf{X}_t,\mu^N_t)]\mathrm{d}\mathbf{W}_t$, where $M(\Gamma_t)$ is a matrix depending on $\Gamma_t$, and for $\varphi=b$ or $\sigma$, we denoted $\vec{\varphi}(t,\mathbf{X}_t,\mu^N_t)$ the vector in $\R^N$ with $i-$th entry $\varphi(t,X^i_t,\mu^N_t)$, and $\mbox{\rm diag}[\vec{\varphi}]$ is the diagonal matrix with diagonal elements defined by the entries of the vector $\vec{\varphi}$. As $\Gamma_t$ has non--negative entries, it follows that $M(\Gamma_t)$ is a diagonally dominant matrix and is therefore invertible. Consequently $\mathbf{X}$ is an It\^o process defined by the drift and the diffusion coefficients $B=(B^i)_{1\le i\le N}$ and $\Sigma=(\Sigma^{i,j})_{1\le i,j\le N}$:
\begin{eqnarray} 
    \mathrm{d}\mathbf{X}_t
    =
    B_t \mathrm{d}t
    +
    \Sigma_t \mathrm{d}\mathbf{W}_t,
    &\!\!\mbox{with}&\!\!
    B_t
    =
    M(\Gamma_t)^{-1}\vec{b}(t,\mathbf{X}_t,\mu^N_t)
    =:
    \mathbf{B}(t,\Gamma_t,\mathbf{X}_t),
    \label{eq:def_coef}\\ &\!\!\mbox{and}&\!\!
    \Sigma_t
    =
    M(\Gamma_t)^{-1}\mbox{\rm diag}[\vec{\sigma}(t,\mathbf{X}_t,\mu^N_t)]
    =:
    \mathbf{\Sigma}(t,\Gamma_t,\mathbf{X}_t).
    \nonumber
\end{eqnarray}
For later use, we isolate the equations defining the coefficients $(B^i)_{1 \le i \le N}$ and $(\Sigma^{i,j})_{1 \le i,j \le N}$:
\begin{eqnarray} \label{def:drift_equation}
    B^i_t
    &=&
    \frac{1}{N} \sum_{j=1}^N \gamma^{i,j}_t B^j_t- \frac{1}{N} \sum_{j=1}^N \gamma^{j,i}_t B^i_t + b(t,X^i_t,\mu^N_t)
\\
\Sigma^{i,q}_t
&=&
\frac{1}{N} \sum_{j=1}^N \gamma^{i,j}_t \Sigma^{j,q}_t- \frac{1}{N} \sum_{j=1}^N \gamma^{j,i}_t \Sigma^{i,q}_t + \sigma(t,X^i_t,\mu^N_t) \mathbf{1}_{q=i}.
\label{def:vol_equation}
\end{eqnarray}

\paragraph*{Deviating player}
For any $[0,1]^N$--valued $\F$--predictable process $\beta:=(\beta^1,\cdots,\beta^N)^\intercal,$ we introduce the deviated matrix strategy defined by substituting $\beta^\intercal$ to the $i-$th line of $\Gamma$:
\begin{align*}
    \Gamma^{-i}(\beta)
    :=
    \Big(({\gamma^{1,\cdot}})^{^\intercal},\cdots,({\gamma^{i-1,\cdot}})^{^\intercal},\beta_t, ({\gamma^{i+1,\cdot}})^{^\intercal},\cdots,({\gamma^{N,\cdot}})^{^\intercal}\Big)^{^\intercal}\;\mbox{where}\;\gamma^{i,\cdot}:=(\gamma^{i,j})_{1 \le j \le N}.
\end{align*}
Following the same argument as in the mean field formulation of the mutual holding problem in Subsection \ref{sect:MFGformulation}, we now introduce the equivalent probability measure $\P^i_{\Gamma,\beta}$ defined by the Radon-Nykodim density:
$$
 \frac{\mathrm{d}\P^i_{\Gamma,\beta}}{\mathrm{d}\P}
 := Z^i_T:=
 e^{\int_0^T \psi^i_t\mathrm{d}W^i_t-\frac12| \psi^i_t|^2\mathrm{d}t},
 ~\mbox{on}~\Fc_T, ~\mbox{with}~\psi^i_t
    :=
    \frac{\mathbf{B}^{i}(t,\Gamma^{-i}_t(\beta_t),\mathbf{X}_t)-\mathbf{B}^{i}(t,\Gamma_t,\mathbf{X}_t)}{\mathbf{\Sigma}^{i,i}(t,\Gamma^{-i}_t(\beta_t),\mathbf{X}_t)},
$$
and we define $\widehat{W}^{i,i}_\cdot:=W^i_\cdot-\int_0^\cdot \psi^i_s \mathrm{d}s$. By the Girsanov Theorem,  the process $\widehat{\mathbf{W}}^i:=(W^1,\cdots,W^{i-1},\widehat{W}^{i,i},W^{i+1},\cdots,W^N)$ is a $\P^i_{\Gamma,\beta}$--Brownian motion. 

We now introduce the process $\mathbf{X}^i:=(X^{i,1},\cdots,X^{i,N})$ defined by initial conditions $X^{i,k}_0=X^k_0$, and the stochastic differential equation:
\begin{eqnarray}\label{def:Xi}
    \mathrm{d}\mathbf{X}^i_t
    &=&
    \mathbf{B}(t,\Gamma^{-i}_t(\beta_t),\mathbf{X}^i_t) \mathrm{d}t
    +
    \mathbf{\Sigma}(t,\Gamma^{-i}_t(\beta_t),\mathbf{X}^i_t) \mathrm{d}\widehat{\mathbf{W}}_t,
\end{eqnarray}
where the functions $\mathbf{B}$ and $\mathbf{\Sigma}$ are defined in \eqref{eq:def_coef}.  Given $\Gamma,$ notice that $\Xbb^i$ is well defined since the maps $\mathbf{B}$ and $\mathbf{\Sigma}$ are Lipschitz in $x$ (see Lemma \ref{lemma:coefficients} below). Also, as $B=(B^k)_{1 \le k \le N}$ and $\Sigma=(\Sigma^{k,q})_{1 \le k,q \le N}$ satisfy \eqref{def:drift_equation} and \eqref{def:vol_equation}, we can then easily check that the $\P^i_{\Gamma,\beta}-$ dynamics of the process $\mathbf{X}^i$ agrees with the mutual holding dynamics of Section \ref{eq_nagents-intuition}:
\begin{eqnarray*}
    \mathrm{d}X^{i,i}_t
    &\!\!\!\!=&\!\!\!\!
    \frac{1}{N} \sum_{j\neq i} \beta^{j}_t \mathrm{d}X^{i,j}_t 
    -
    \frac{1}{N} \sum_{j\neq i} \gamma^{j,i}_t \mathrm{d}X^{i,i}_t
    +
    \mathrm{d}P^{i,i}_t,
\\
    \mathrm{d}X^{i,k}_t
    &\!\!\!\!=&\!\!\!\!
    \frac{1}{N} \sum_{j=1}^N \gamma^{k,j}_t \mathrm{d}X^{i,j}_t 
    -
    \frac{1}{N} \bigg(\beta^{k}_t + \sum_{j \neq i} \gamma^{j,k}_t\bigg)\mathrm{d}X^{i,k}_t 
    +\mathrm{d}P^{i,k}_t
    ~\mbox{for}~k \neq i,
    \\
    \mathrm{d}P^{i,k}_t
    &\!\!\!\!=&\!\!\!\!
    b(t,X^{i,k}_t,\mu^{i,N}_t) \mathrm{d}t
    +
    \sigma(t,X^{i,k}_t,\mu^{i,N}_t) \mathrm{d}\widehat{W}^{i,k}_t,
    ~k=1,\ldots,N,
    ~\mbox{with}~
    \mu^{i,N}_t:=\frac1N\sum_{j=1}^N \delta_{X^{i,j}_t}.
\end{eqnarray*}
Consequently, we may define the $i-$th player reward from using the strategy $\beta$, given that the remaining agents stick to their strategies $\gamma^{k,\cdot}$, $k\neq i$, by
\begin{align*}
    J_i(\Gamma^{-i}(\beta))
    :=
    \E^{\P^i_{\Gamma,\beta}}[U(X^{i,i}_T)],
    ~~i=1,\ldots,N.
\end{align*}

\begin{Definition}\label{def:N-Nash}{\rm(Approximate Nash equilibrium)}
    For $\varepsilon \ge 0,$ we say that $\Gamma$ is an $\varepsilon$--Nash equilibrium if: 
    \begin{eqnarray*}
        J_i(\Gamma) &\ge&  \sup_{\beta} J_i(\Gamma^{-i}(\beta))  - \varepsilon,
        ~~\mbox{for all}~~i \in \{1,\cdots,N\}.
    \end{eqnarray*}
\end{Definition}

\begin{Remark}
{\rm 
The last definition may seem very involved when compared to our (formal) problem description in {\rm Section \ref{sec:mutualholding}}. It is however tailor-maid so as to mimick the spirit of our mean field formulation in {\rm Subsection \ref{sect:MFGformulation}}. Notice in particular that it is slightly different from the natural formal description of the deviating agent problem in \eqref{N-natural}, where the control of the $i$--th agent acts both on the drift and the diffusion coefficients, thus preventing from expressing the problem in terms of optimizing some change of measure. The discussion preceding Definition {\rm \ref{def:N-Nash}} shows that the deviating equity value process $\mathbf{X}^i$ defined in \eqref{def:Xi} has the same $\P^i_{\Gamma,\beta}$--dynamics than the process $\mathbf{X}$ involved in \eqref{eq_nagents-intuition}. Consequently, the corresponding $N$--players mutual holding problem are {\it essentially} the same.
}
\end{Remark}

\paragraph*{From MFG solution to finite population approximate Nash equilibria}
We now use the structure of the MFG solution in Theorem \ref{thm:main} in order to construct approximate Nash equilibria as in Definition \ref{def:N-Nash}.
 
Let $B:[0,T] \x \R \x \Pc(\R) \to \R$ and $c:[0,T] \x \Pc(\R) \to \R_+$ be the maps given in Theorem \ref{thm:main}, and define for $\mathbf{x}=(x^1,\ldots,x^N)\in\R^N$:
\begin{eqnarray}\label{piN}
    \pi(t,x^i,m^N)
    :=\pi^i(t,\mathbf{x}):=
    \mathbf{1}_{\{B(t,x^i,m^N) \ge 0 \}}
    &\mbox{where}&
    m^N:=\frac{1}{N} \sum_{j=1}^N \delta_{x^j}.
\end{eqnarray}
and for $i,j=1,\cdots, N$:
\begin{align*}
    \Sigma^{i,j}(t,\mathbf{x})
    :=
    \frac{\sigma(t,x^i,m^N) \mathbf{1}_{\{i=j\}}
            \!+\!\frac{1}{N}A^j(t,\mathbf{x})\sigma(t,x^q,m^N)}
           {1+\pi(t,x^i,m^N)},
    ~A^j(t,\mathbf{x})
    :=
    \frac{  \frac{\pi^j(t,\mathbf{x})}{1+\pi^j(t,\mathbf{x})}}
           {1\!-\! \frac{1}{N} \sum_{k=1}^N \frac{\pi^k(t,\mathbf{x})}{1+\pi^k(t,\mathbf{x})}}.
\end{align*}
By \citeauthor[Section 2--Part 6--Theorem 1]{KrylovControlledDiffusion} \cite[Section 2--Part 6--Theorem 1]{KrylovControlledDiffusion}, there exists a weak solution $(X^1,\cdots,X^N)$ to the stochastic differential equation:
\begin{eqnarray}
    \mathrm{d}X^i_t
    &=&
    B(t,X^i_t,\mu^N_t) \mathrm{d}t
    +
    \sum_{j=1}^N\Sigma^{i,j}(t,X^i_t,\mu^N_t) \mathrm{d}W^j_t\;\;\mbox{where}\;\;\mu^N_t:=\frac{1}{N} \sum_{i=1}^N \delta_{X^i_t}
    \nonumber\\
    &=&
    \frac{1}{N} \sum_{j=1}^N \pi^{i,j}_t \mathrm{d}X^j_t 
    -
    \frac{1}{N} \sum_{j=1}^N \pi^{j,i}_t \mathrm{d}X^i_t
    +
    b(t,X^i_t,\mu^N_t) \mathrm{d}t
    +
    \sigma(t,X^i_t,\mu^N_t) \mathrm{d}W^i_t,
    \label{def:nash_equilibria}
\end{eqnarray}
where the last equality follows by direct verification, and where we used the notation:  
\begin{eqnarray}\label{PiN}
\pi^{i,j}_t:=\pi^j_t:=\pi(t,X^j_t,\mu^N_t)
&\mbox{and}&
\Pi^N:=(\pi^{i,j})_{1 \le i,j \le N},
~~\mbox{for all}~~N\ge 1.
\end{eqnarray}
We are now able for the statement of our last result.

\begin{Theorem} \label{thm:nash_equilibria}
For all $N\ge 1$, the mutual holding strategy $\Pi^N$ is an $\varepsilon_N$--Nash equilibrium, for some sequence $\varepsilon_N\ge 0$ satisfying $\lim_{N \to \infty} \varepsilon_N =0$.
\end{Theorem}

}

{\color{black}
\section{A first step towards numerical illustration} \label{sec:numerical} 

In this paragraph we illustrate briefly our main result in the simple situation where the provisions process, which represents the idiosyncratic risk, is defined by an Ornstein--Uhlenbeck process:
\begin{align*}
    b(t,x,m)
    =
    b(t,x)
    =
    \theta (\overline{m}-x)\;\;\mbox{and}\;\;
    \sigma(t,x,m)
    =
    \overline{\sigma},
\end{align*}
with given parameters $\theta > 0$ and $(\overline{\sigma}, \overline{m}) \in \R \x \R$. 
In order to emphasize the behavior generated by the optimal control, we aim at visualizing the distribution of the optimal path $X^\star$ and the distribution of the path corresponding to the non--interacting situation. In other words, we want to compare the optimal behavior with the no cross--holding situation where the equity value process is simply given by the provisions $P$, see \eqref{eq_nagents-intuition}. 

Let us recall that the solution of the mean field game of mutual holding provided in Theorem \ref{thm:main} does not depend on the choice of the utility function $U$, as long as it is nondecreasing and Lipschitz.

Given the involved nature of the equilibrium mean field McKean-Vlasov SDE with coefficients $B$ and $\Sigma$, as defined in Theorem \ref{thm:main} (i), a natural method for the numerical illustration would be based on the approximation of the mean field dynamics by the corresponding interacting particles system. The analysis of this method appears to be very challenging in our context because of the non--regularity of the problem and the equation defining the function $c.$ We therefore refrain from addressing a complete numerical study of this method, and 
\begin{center}
{\it we focus on a simple numerical illustration which highlights the main effects underlying our main results. The accurate numerical analysis is left for future research. }
\end{center}
Instead of the equilibrium mean field McKean-Vlasov SDE with coefficients $B$ and $\Sigma$, we consider a one--step illustration based on a one-step Euler approximation of the provisions and the equilibrium equity dynamics processes:
\b*
    P_T
    &=&
    X_0 
    +
    b(0,X_0) \Delta 
    +
    \overline{\sigma} \sqrt{\Delta} Z,
    \;\;\mbox{where}\;\;X_0 \sim \mu_0=\mathcal{N}(\mu,\sigma^2),
    \\
    X^\star_T
    &=&
    X_0 
    +
    B(0,X_0,\mu_0) \Delta 
    +
    \Sigma(0,X_0,\mu_0) \sqrt{\Delta} Z.
\e*
with $Z \sim \mathcal{N}(0,1),$ and $\Delta=T.$ We emphasize that by taking $\Delta=T$, we are not considering the Euler discretization of the McKean-Vlasov equilibrium SDE, as one typically would like to analyze the convergence towards the solution of our problem. Again, we are fixing $\Delta=T$, and we are modestly aiming at exploring the distribution induced by the above simplified dynamics.

Given the particular structure of our problem, the constant $c$ appearing in the optimal drift $B$ is defined by $H(c)=0$ where
\begin{align*}
    H(x)
    :=
    x
    -
    \frac{1}{2} \theta \sigma^2 f_0\Big(\frac{x}{\theta}+\overline{m} \Big) 
    -
    \frac{1}{2} (x-\theta(\mu-\overline{m})) F_0\Big(\frac{x}{\theta}+\overline{m}\Big),
\end{align*}
with $f_0$ and $F_0$ the density and the cumulative distribution functions of $X_0$. As the provisions process $P$ are defined by an Ornstein--Uhlenbeck process, we chose to start with the corresponding invariant distribution, namely $\mu_0$ is a Gaussian distribution with mean $\mu=\overline{m}$ and variance $\sigma^2=\frac{\overline{\sigma}^2}{2 \theta}$. 
\begin{figure}[!h]
    \centering
    \includegraphics[width=7cm]{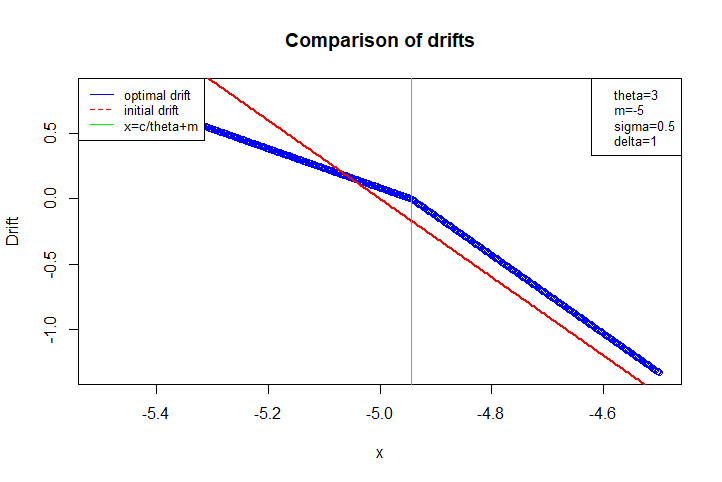}
    \vspace{-5mm}
    \caption{\small \it The equity process drift $B$ as an averaging of the provisions drift $b$}
    \label{fig:B-b}
\end{figure}

\noindent{\bf Equilibrium dynamics of the equity process.} Our main result in Theorem \ref{thm:main} states that the diffusion $\Sigma$ of the equity process at equilibrium is either equal to the provisions process diffusion, or cut by half, depending on the sign of its drift coefficient $B$. This illustrates the diversification effect of the optimal mutual holding between our continuum population of firms. Similarly the drift $B$ is related to the drift $b$ of the provisions process by the relation $B^-=(b+c)^-$ and $B^+=\frac12(b+c)^+$ with $c$ defined in Theorem \ref{thm:main}. Figure~\ref{fig:B-b} shows that the equilibrium equity value drift represents an averaging (in some sense) of the provisions drift $b$.

\vspace{3mm}
\noindent{\bf Effect of the provisions mean $\overline{m}$.} Figure \ref{fig:mean} compares the distributions of the provisions and the equilibrium equity value. The density of the equity value exhibits smaller variance than that of the provisions process. This illustrates the diversification effect of the cross holding between individual firms. A remarkable effect is that the density is pushed to the right when the provisions process has a negative mean. This indicates that the optimal cross holding policy of firms with positive drift induces an overall increase of the mean of the population's equity value.
\begin{figure}[!h]
    \centering
    \begin{subfigure}[b]{0.4\textwidth}
        \includegraphics[width=\textwidth]{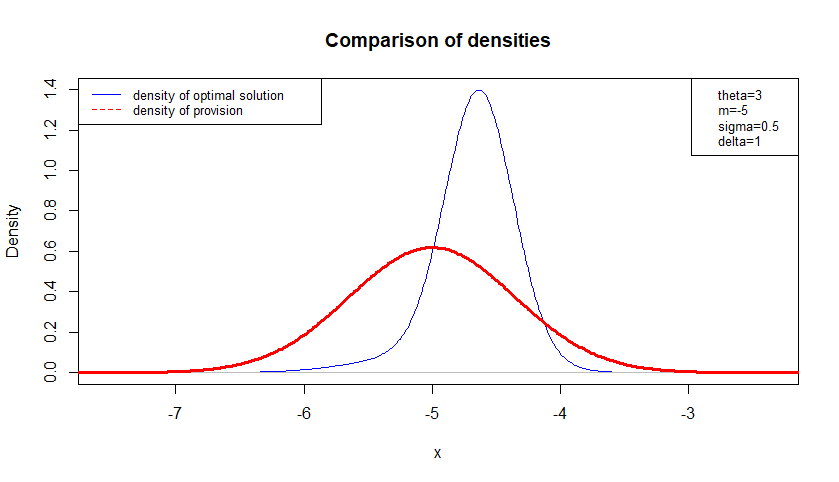}
    \end{subfigure}
    \begin{subfigure}[b]{0.4\textwidth}
        \includegraphics[width=\textwidth]{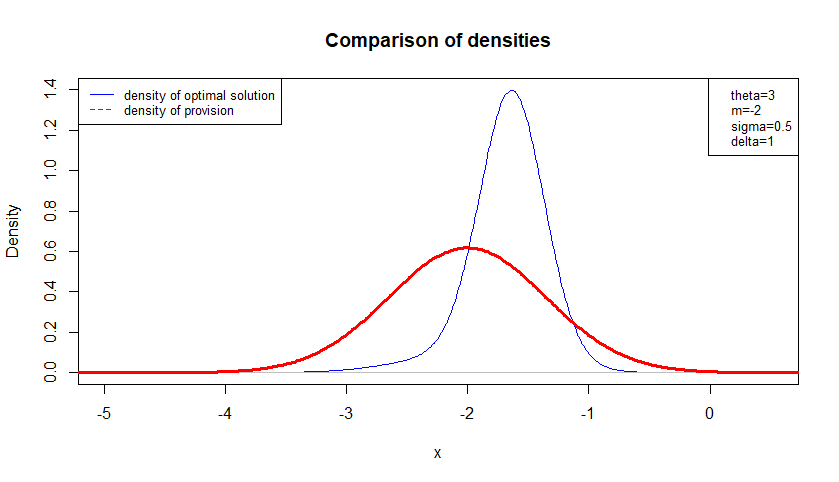}
    \end{subfigure}
    \begin{subfigure}[b]{0.4\textwidth}
        \includegraphics[width=\textwidth]{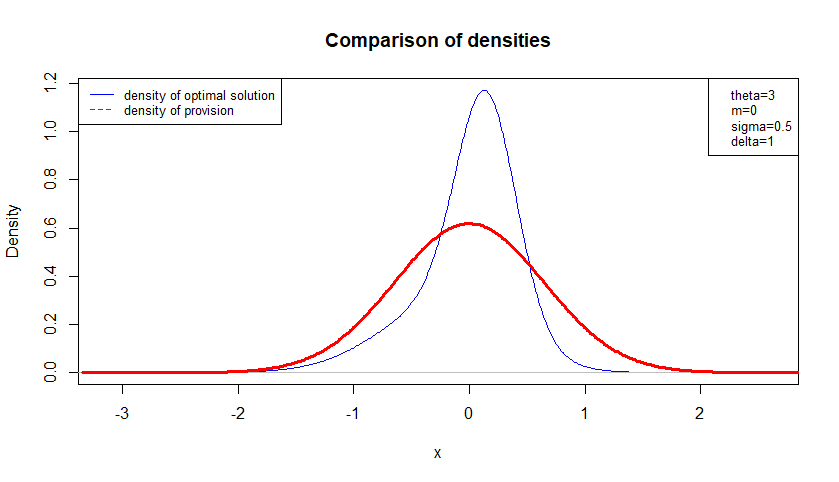}
    \end{subfigure}
    \begin{subfigure}[b]{0.4\textwidth}
        \includegraphics[width=\textwidth]{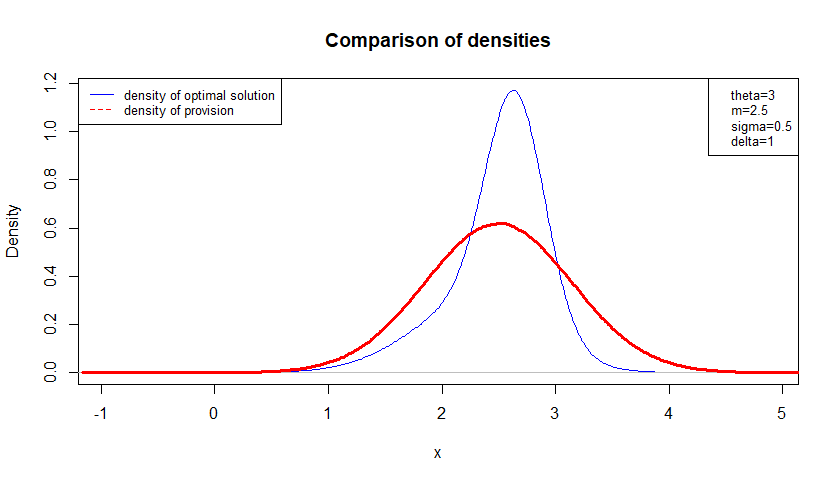}
    \end{subfigure}
    \begin{subfigure}[b]{0.4\textwidth}
        \includegraphics[width=\textwidth]{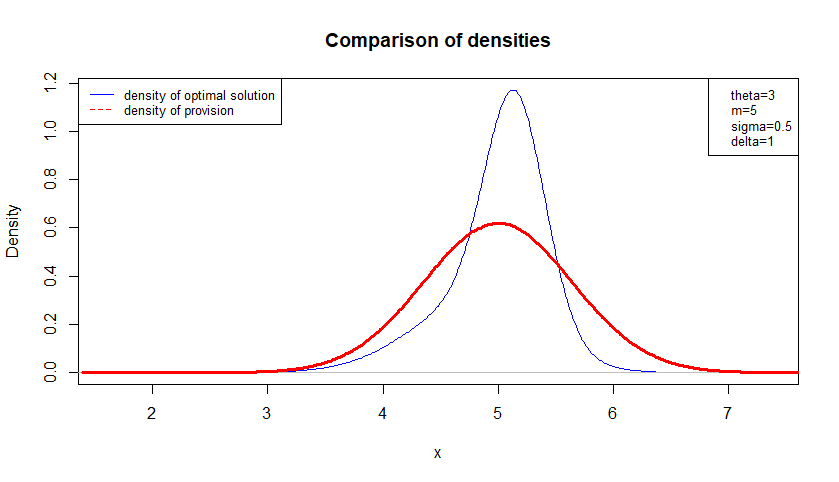}
    \end{subfigure}
\vspace{-5mm}
\caption{\small \it  Effect of $\overline{m}$: equilibrium equity value has smaller variance; moreover under a negative mean of the provisions, the equilibrium equity value process has a significantly larger mean.}
\label{fig:mean}
\end{figure}

\vspace{3mm}
\noindent{\bf Effect of the provisions volatility $\overline{\sigma}$.} Figure \ref{fig:vol} compares the distributions of the provisions and the equilibrium equity value. We set the mean $\overline{m}$ to zero. As expected from the result of Theorem \ref{thm:main}, the magnitude of the volatility reduction does not depend on the range of $\bar\sigma$, as $\Sigma$ either coincides with $\sigma$ or $\frac12\sigma$. Our numerical illustration shows an asymmetry effect on the distribution which is more visible under large volatility.
\begin{figure}[!h]
    \centering
    \begin{subfigure}[b]{0.4\textwidth}
        \includegraphics[width=\textwidth]{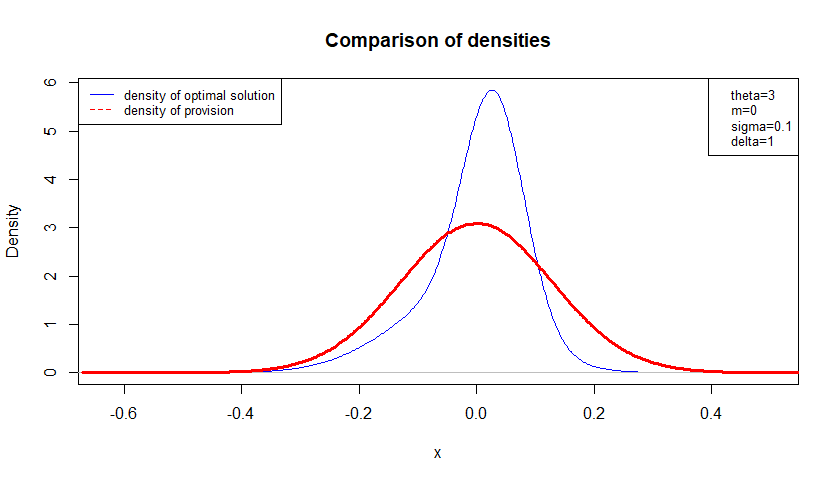}
    \end{subfigure}
    \begin{subfigure}[b]{0.4\textwidth}
        \includegraphics[width=\textwidth]{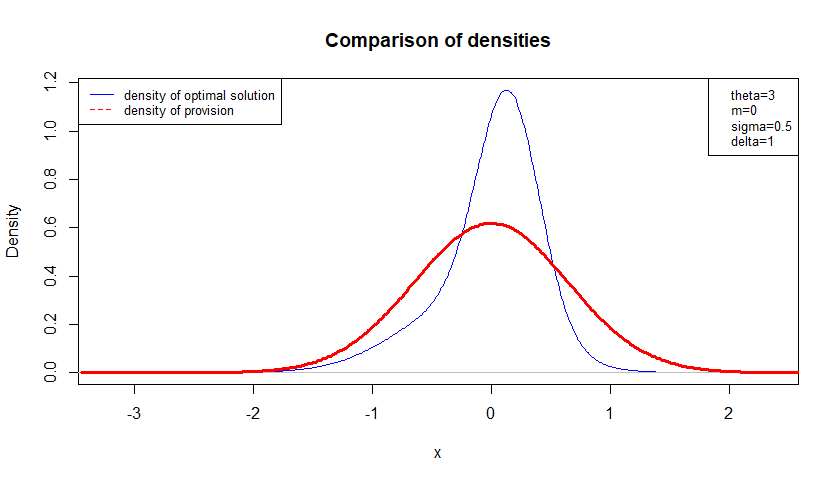}
    \end{subfigure}
    \begin{subfigure}[b]{0.4\textwidth}
        \includegraphics[width=\textwidth]{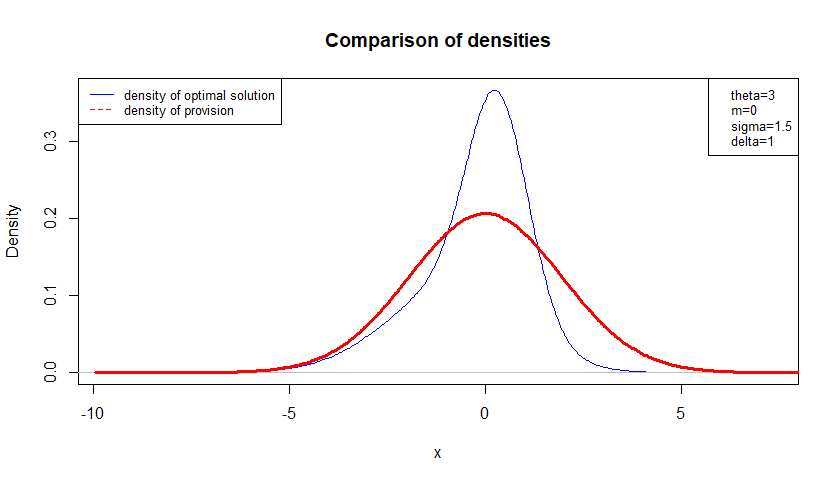}
    \end{subfigure}
    \vspace{-5mm}
    \caption{\small\it Effect of $\overline{\sigma}$: the variance of the equilibrium equity value relative to that of the provisions seems to have a similar magnitude for various values of $\overline{\sigma}$, but exhibits an asymmetry effect.}
    \label{fig:vol}
\end{figure}

\vspace{3mm}
\noindent{\bf Effect of the provisions mean reversion speed $\theta$.} As the speed of mean reversion get larger, Figure \ref{fig:meanreversion} below shows a remarkable asymmetry effect between the negative support and the positive one. Individual firms with negative drift essentially conserve their situation, while individual firms with positive drift tend to be more concentrated and thus have a smaller conditional variance.
\begin{figure}[!h]
    \centering
    \begin{subfigure}[b]{0.35\textwidth}
        \includegraphics[width=\textwidth]{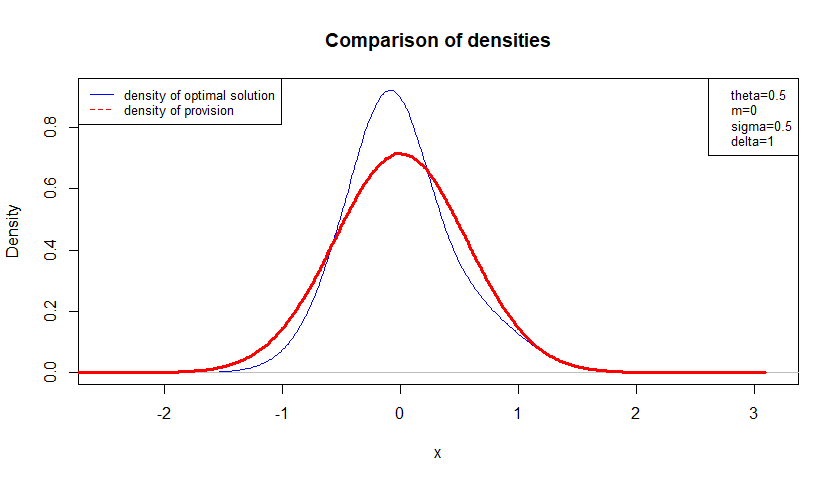}
    \end{subfigure}
    \begin{subfigure}[b]{0.35\textwidth}
        \includegraphics[width=\textwidth]{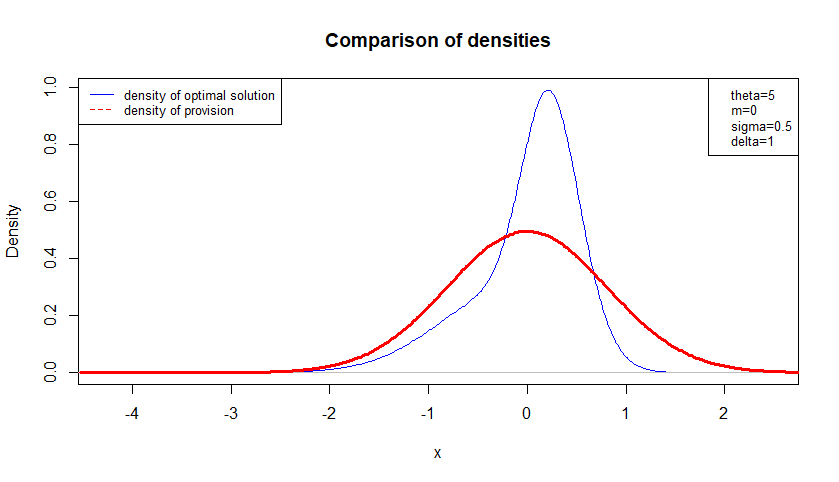}
    \end{subfigure}
    \begin{subfigure}[b]{0.35\textwidth}
        \includegraphics[width=\textwidth]{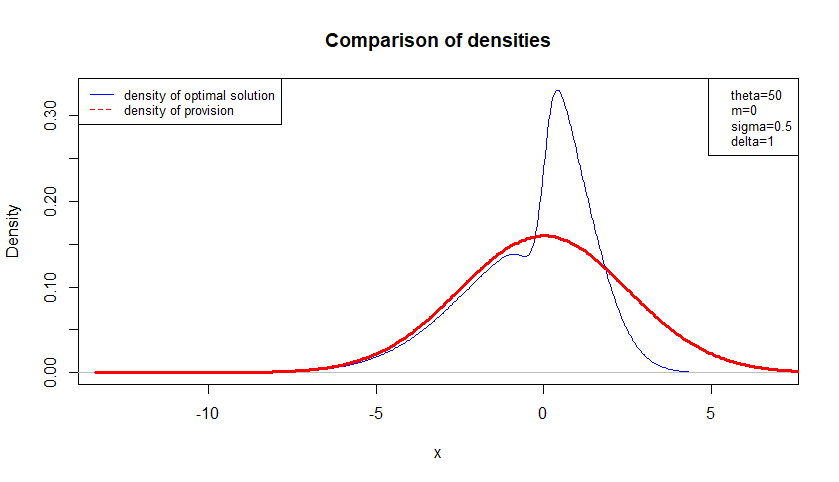}
    \end{subfigure}
    \begin{subfigure}[b]{0.35\textwidth}
        \includegraphics[width=\textwidth]{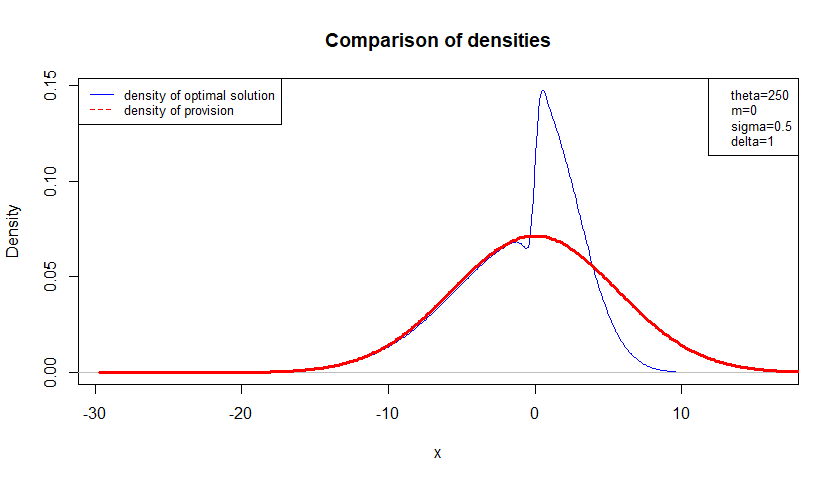}
    \end{subfigure}
    \begin{subfigure}[b]{0.35\textwidth}
        \includegraphics[width=\textwidth]{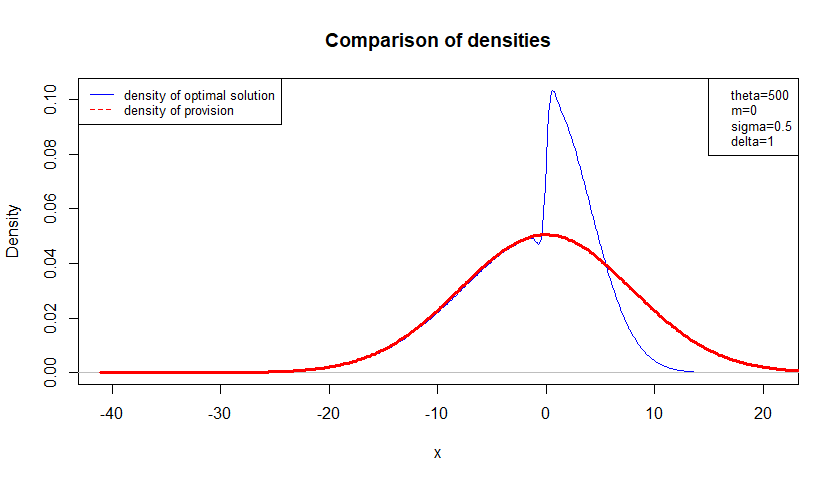}
    \end{subfigure}
    \caption{\small \it Effect of $\theta$: under large mean reversion the density of the equity value inherits that of the provisions on the negative support, while the positive part benefits from diversification. }
    \label{fig:meanreversion}
\end{figure}

}

\section{Deriving a solution of the mean field game equilibrium}
\label{sec:proofmfg}

Let $\mu\in\Pc_\Sc(\pi^*)$ be the distribution of a solution of the mean field SDE \eqref{MFG-SDE}, whose existence is guaranteed by Theorem \ref{thm:SDE}. In order to prove that $(\pi^*,\mu)$, is a mean field game equilibrium, we consider the individual representative agent optimization problem
 \b*
 V_0
 =
 \sup_{\beta\in\Ac} J_{\pi^*,\mu}(\beta),
 &\mbox{where}&
 J_{\pi^*,\mu}(\beta)
 :=
 \E^{\P_{\pi^*,\mu}^\beta}\big[U(X_T)\big],
 \e*
and the dynamics of the state process $X$ in terms of the $\P_{\pi^*,\mu}^\beta$--Brownian motion $W^{\pi^*,\mu,\beta}$ are given by:
 \b*
 \mathrm{d}X_t
 &=&
 \frac{\big[b(t,X_t,\mu_t)+\int_\R\beta(t,X_t,y)B(t,y,\mu_t)\mu_t(\mathrm{d}y)\big]\mathrm{d}t+\sigma(t,X_t,\mu_t)\mathrm{d}W^{\pi^*,\mu,\beta}_t}
        {1+\int_\R\pi^*(t,y,X_t)\mu_t(\mathrm{d}y)}
 \\
 &=&
 \frac{\big[b(t,X_t,\mu_t)+\int_\R\beta(t,X_t,y)B(t,y,\mu_t)\mu_t(\mathrm{d}y)\big]\mathrm{d}t+\sigma(t,X_t,\mu_t)\mathrm{d}W^{\pi^*,\mu,\beta}_t}
        {1+\1_{\{B(t,X_t,\mu_t)\ge 0\}}},
 ~~\P\mbox{--a.s.}
 \e*
with $B=\frac12(b+c)^+-(b+c)^-$, by Theorem \ref{thm:main} (i). In order to write the corresponding HJB equation in its backward SDE form, we introduce the corresponding Hamiltonian defined for all $(t,x,m)\in [0,T]\times\R\times\Pc_2(\R)$ by
 \b*
 H_t(x,m,z)
 &:=&
 \frac{ z\,b(t,x,m)+\sup_{\beta(.)} z\int_\R\beta(y)B(t,y,m)m(\mathrm{d}y)}
        {1+\1_{\{B(t,x,m)\ge 0\}}}
 \\
 &=&
 \frac{ z\,b(t,x,m)+\int_\R(zB(t,y,m))^+m(\mathrm{d}y)}
        {1+\1_{\{B(t,x,m)\ge 0\}}},
 \e*
with maximizer 
 $$
 \widehat\beta(t,y,m,z):=\1_{\{zB(t,y,m)\ge 0\}}.
 $$
Consider the backward stochastic differential equation
 \be\label{BSDE0}
 Y_T=U(X_T),
 &\mbox{and}&
 \mathrm{d}Y_t=Z_t \mathrm{d}X_t-H_t(X_t,\mu_t,Z_t)\mathrm{d}t, ~~t\in[0,T],~~\mu\mbox{--a.s.}
 \ee
Substituting the expression of the Hamiltonian $H$, and the equilibrium dynamics of $X$ corresponding to the optimal holding $\pi^*$ of Theorem \ref{thm:main} (ii), we rewrite the dynamics of the process $Y$ of the backward SDE \eqref{BSDE} in terms of the Brownian motion $W^\mu=W$ as:
 $$
 \mathrm{d}Y_t
 =
 \frac{ \int_\R \big[Z_tB(t,y,\mu_t)^+-(Z_tB(t,y,\mu_t))^+ \big]\mu_t(\mathrm{d}y)\;\mathrm{d}t+Z_t\sigma(t,X_t,\mu_t)\mathrm{d}W_t}
        {1+\1_{\{B(t,X_t,\mu_t)\ge 0\}}}
 $$
which induces, after substituting the expression of $B$ in Theorem \ref{thm:main} (i):
 \be\label{BSDE}
 Y_T=U(X_T),~~\mbox{and}~~
 \mathrm{d}Y_t
 &=&
 F_t(X_t,Z_t)\mathrm{d}t
 +Z_t\Sigma(t,X_t,\mu_t)\mathrm{d}W_t,~\P\mbox{--a.s.}
 \ee
where the generator $F$ is given by
 \b*
 F_t(x,z)
 &:=&
 \frac{ \int_\R[zB^+-(z\;B)^+](t,y,\mu_t)\mu_t(\mathrm{d}y)}
        {1+\1_{\{B(t,x,\mu_t)\ge 0\}}}
 ,~~t\le T,~\mbox{and}~x,z\in\R.  
 \e*

\begin{Lemma}\label{lem:BSDE}
The backward SDE \eqref{BSDE} has a unique solution $(Y,Z)\in\H^2(\P)\times\H^2(\P)$, with 
 \b*
 Z\ge 0
 &\mbox{and}&
 Y_t = \E^\P\big[U(X_T)|\Fc_t\big],~t\le T,
 \e*
\end{Lemma}

\proof Recall that $X$ is square integrable by Theorem \ref{thm:SDE}. Then,  it follows from the Lipschitz property of $U$ that the terminal condition $U(X_T)\in\L^2(\P)$. Also, notice that $F$ is Lipschitz in $z$ by Lemma \ref{lem:cvarphi}, and that $F_t(X_t,0)=0$. We have thus verified the standard conditions of \citeauthor*{pardoux1990adapted}
\cite{pardoux1990adapted} for the existence of a unique solution for the backward SDE \eqref{BSDE}. 

We next prove that $Z \ge 0$, $\P$--a.e. Notice that under the guess $Z\ge 0$, the generator of the last backward SDE vanishes, and we are reduced 
to the candidate solution of 
    \b*
    \overline{Y}_T=U(X_T),
    &\mbox{and}&
    \mathrm{d}\overline{Y}_t
    =
    \overline{Z}_t\Sigma(t,X_t,\mu_t)
        \,\mathrm{d}W_t,~~\P\mbox{--a.s.}
    \e*
The $\overline{Y}$--component of the unique solution of this equation is the $\H^2(\mu)$--process $\overline{Y}_t=\E^{\P}\big[U(X_T)|\Fc_t\big]$, $t\le T$, and we now prove that $\overline{Z}\ge 0$, $\mathrm{d}t\otimes \mathrm{d}\P$--a.e. This would imply that $(\overline{Y},\overline{Z})$ is also the solution of the backward SDE \eqref{BSDE}, and that $(Y,Z)=(\overline{Y},\overline{Z})$ by uniqueness.

\medskip    
\no {\it The smooth case.}
Let us first assume that $\Sigma$ admits bounded derivative w.r.t. $x$. By the expression of $B$ in Theorem \ref{thm:main} (i) together with Lemma 
\ref{lem:cvarphi}, $B$ is uniformly Lipsichtz in $x$, and therefore has a 
bounded weak derivative in $x$ by Rademacher Theorem. By a direct density 
argument, we may use a slightly extended version of the Clark-Ocone formula, see e.g. \citeauthor*{nualart2006malliavin} \cite[Proposition 1.5]{nualart2006malliavin}, to obtain the explicit representation of the $\overline{Z}-$component:
\begin{align*}
    \overline{Z}_t=\vartheta (t,X_t) \Sigma(t,X_t),\;\P\mbox{--a.e.}\;\mbox{for all}\;t \in [0,T],\;\mbox{with}\;\vartheta(t,x):=\E^{\P}[U'(X^{t,x}_T)J^{t,x}_T],
\end{align*}
where $U'$ is the weak derivative of the Lipschitz function $U$, and
    \b*
        X^{t,x}_s
        &=& 
        x+\int_t^s B(r,X^{t,x}_r,\mu_r) \mathrm{d}r 
        + \int_t^s\Sigma(r,X^{t,x}_r,\mu_r) \mathrm{d}W_r
         \\ 
 J^{t,x}_s
 &=& 
 1+\int_t^s \partial_x B(r,X^{t,x}_r,\mu_r) J^{t,x}_r \mathrm{d}r 
 + \int_t^s\partial_x\Sigma(r,X^{t,x}_r,\mu_r) J^{t,x}_r \mathrm{d}W_r.
    \e*
Notice the tangent process $J^{t,x}$ is positive and $U$ is nondecreasing, we deduce that $\overline{Z}_t \ge 0,$ $\P$--a.e. for all $t \in [0,T]$.

\medskip
\no {\it The general case.} We now relax the smoothness of $\Sigma$ assumed in the previous step. By an easy extension of \citeauthor*{MFD-2020} \cite[Lemma 4.2.]{MFD-2020}, for each $n \ge 1,$ there exists smooth $\Sigma^n,$ with $X^n$ solution of 
    \begin{align*} 
        X^n_\cdot=X_0 + \int_0^\cdot B(t,X^n_t,\mu_t) \mathrm{d}t + \int_0^\cdot \Sigma^n(t,X^n_t,\mu_t) \mathrm{d}W_t\;\P\;\mbox{--a.e. with}\;\mu^n=\Lc^{\P}(X^n),
    \end{align*}
    such that $\lim_{n \to \infty} \mu^n=\mu$ in $\Wc_2.$
    Let $(\overline{Y}^n,\overline{Z}^n)$ be the unique solution of
\begin{align*}
    \overline{Y}^n_t
    =
    U(X^n_T) - \int_t^T \overline{Z}^n_s \Sigma^n(s,X^n_s,\mu_s) \mathrm{d}W_s,\;\mbox{for all}\;t \in [0,T],\;\P\mbox{--a.e.}
\end{align*}  
By the previous smooth case, we know that $Z^n \ge 0$. In order to complete the proof, we now show that 
    \begin{eqnarray*}
        \E^{\P} \bigg[\int_0^T \overline{Z}_t\;\varphi_t \;\mathrm{d}t \bigg] 
        \ge 0
        &\mbox{for all non--negative valued bounded $\F^X-$predictable}&
        \varphi,
    \end{eqnarray*}
    where $\F^X:=\big( \sigma\{X_s,\;s \le t \} \big)_{t \in [0,T]}.$
    It is enough to show this result for processes of type $\varphi_t=\phi_t(X_{t \wedge \cdot}) \Sigma(t,X_t,\mu_t)^2$ s.t. there exists $t_0=0 \le t_1 \le \dots \le t_q=T,$ with $\phi_t(X_{t \wedge \cdot})=\phi_{t_k}(X_{t_k \wedge \cdot})$ for each $t \in [t_k,t_{k+1})$, and for all 
$t \in [0,T],$ $\phi(t,\cdot)$ is continuous. Denoting simply $\Sigma_t:=\Sigma(t,X_t,\mu_t)$, we compute that 
     \begin{align*}
        &\E^{\P} \bigg[\int_0^T \overline{Z}_t\phi_t(X_{t \wedge \cdot}) \Sigma_t^2 \mathrm{d}t \bigg] 
        =
        \E^{\P} \bigg[\int_0^T \overline{Z}_t\Sigma_t\mathrm{d}W_t \;\int_0^T\phi(t, X_{t \wedge \cdot}) \Sigma_t \mathrm{d}W_t \bigg]
        \\
        &=
        \E^{\P} \bigg[\big( U(X_T) - \overline{Y}_0 \big) \;\int_0^T\phi(t, X_{t \wedge \cdot}) \Sigma_t \mathrm{d}W_t \bigg]
        \\
        &=
        \E^{\P} \bigg[ U(X_T) \;\int_0^T\phi(t, X_{t \wedge \cdot}) \Sigma_t \mathrm{d}W_t \bigg]
        \\
        &=
        \E^{\P} \bigg[ U(X_T) \;\int_0^T\phi(t, X_{t \wedge \cdot}) \Big( 
\mathrm{d}X_t- B(t,X_t,\mu_t) \mathrm{d}t \Big)\bigg]
        \\
        &=
        \E^{\P} \bigg[ U(X_T) \;\sum_{k=0}^{q-1}\phi(t_k, X_{t_k \wedge 
\cdot}) \Big( X_{t_{k+1}}-X_{t_k}- \int_{t_k}^{t_{k+1}} B(t,X_{t},\mu_{t}) \mathrm{d}t \Big)\bigg]
        \\
        &=
        \lim_{n \to \infty}\E^{\P} \bigg[ U(X^n_T) \;\sum_{k=0}^{q-1}\phi(t_k, X^n_{t_k \wedge \cdot}) \Big( X^n_{t_{k+1}}-X^n_{t_k}- \int_{t_k}^{t_{k+1}} B(t,X^n_{t},\mu_{t}) \mathrm{d}t  \Big)\bigg]
        \\
        &=
        \lim_{n \to \infty} \E^{\P} \bigg[ U(X^n_T) \;\int_0^T\phi(t, X^n_{t \wedge \cdot}) \Big( \mathrm{d}X^n_t- B(t,X^n_t,\mu_t) \mathrm{d}t \Big)\bigg],
    \end{align*}
by the convergence in distribution of $X^n$ towards $X$, and the fact that $\phi$ is a simple process. Then:
\begin{eqnarray*}
        \E^{\P} \bigg[\int_0^T \overline{Z}_t\phi_t(X_{t \wedge \cdot}) \Sigma_t^2 \mathrm{d}t \bigg]
        &\!\!\!\!=&\!\!\!\!
        \lim_{n \to \infty} \E^{\P} \bigg[ \big(U(X^n_T)-\overline{Y}^n_0 
\big) \;\int_0^T\phi(t, X^n_{t \wedge \cdot}) \Sigma^n(t,X^n_t,\mu_t) \mathrm{d}W_t\bigg]
        \\
        &\!\!\!\!=&\!\!\!\!
        \lim_{n \to \infty} \E^{\P} \bigg[\int_0^T \overline{Z}^n_t\phi(t, X_{t \wedge \cdot}) \Sigma^n(t,X^n_t,\mu_t)^2 \mathrm{d}t \bigg] \;\ge\;0. 
    \end{eqnarray*}
    As $\overline{Z}$ is an $\F^X-$predictable process, by taking $\varphi_t=\mathbf{1}_{\overline{Z}_t \le 0},$ we find $\overline{Z}_t\mathbf{1}_{\overline{Z}_t \le 0}=0$ $\mathrm{d}t \otimes \mathrm{d}\P$--a.e. We deduce that $\overline{Z}_t \ge 0$ $\mathrm{d}t \otimes \mathrm{d}\P$--a.e. We can conclude as in the smooth situation i.e. $(\overline{Y},\overline{Z})=(Y,Z)$ and $Z \ge 0.$
\ep

The next result shows that the pair $(\pi^*,\mu)$ is a mean field game equilibrium of the mutual holding problem in the sense of Definition \ref{def:MFG}, thus completing the proof of Theorem \ref{thm:main}.

\begin{Proposition} \label{prop:optimization}
We have $\E^\P[U(X_T)] = J_{\pi^*,\mu}(\pi^*) \ge J_{\pi^*,\mu}(\beta)$ 
for all $\beta\in\Ac$. 
\end{Proposition}

\proof For an arbitrary $\beta\in\Ac$, we compute that
 \b*
 J_{\pi^*,\mu}(\beta)
 &=&
 \E^{\P_{\pi^*,\mu}^\beta}\big[ U(X_T) \big]
 \;=\;
 Y_0
 + \E^{\P_{\pi^*,\mu}^\beta}\bigg[\int_0^T Z_t \mathrm{d}tX_t-H_t(X_t,\mu_t,Z_t) \mathrm{d}t \bigg],
 \e*
with $Y_t=\E^\P[U(X_T)|\Fc_t]$, for all $t\in[0,T]$, by Lemma \ref{lem:BSDE}. Let $(\tau^n)_n$ be a localizing sequence for the $\P_{\pi^*,\mu}^\beta-$local martingale $\int_0^. Z_t\Sigma(t,X_t,\mu_t)\mathrm{d}W^{\pi^*,\mu,\beta}$. Then, it follows by direct substitution of the the dynamics of $X$ in terms of the $\P_{\pi^*,\mu}^\beta-$Brownian motion:
  \b*
 \E^{\P_{\pi^*,\mu}^\beta}\big[ Y_{T\wedge\tau^n} \big]
 &=&
 Y_0
 + \E^{\P_{\pi^*,\mu}^\beta}\bigg[\int_0^{T\wedge\tau^n}\!\!\Big(Z_t\frac{ b(t,X_t,\mu_t)+\int_\R\beta(t,X_t,y)B(t,y,\mu_t)\mu_t(\mathrm{d}y)}
    {1+\1_{\{B(t,X_t,\mu_t)\ge 0\}}}
    \\
    &&\hspace{80mm} - H_t(X_t,\mu_t,Z_t)                     \Big)\mathrm{d}t                         \bigg]
 \\
 &\le&
 Y_0 \;=\; \E^\P\big[U(X_T)\big],
 \e*
where the last inequality follows from the definition of $H$. Notice that 
the process $\psi$, defining the density of $\P_{\pi^*,\mu}^\beta$ with respect to $\mu$, is bounded. Then, as $\sup_{t\le T} |Y_t|\in L^2(\P)$, it follows from H\"older's inequality that $\sup_{t\le T} |Y_t|\in L^1(\P_{\pi^*,\mu}^\beta)$. We may then use the dominated convergence theorem to 
obtain that 
 \b*
 J_{\pi^*,\mu}(\beta)
 &=&
 \lim_{n\to\infty}\E^{\P_{\pi^*,\mu}^\beta}\big[ Y_{T\wedge\tau^n} \big]
 \;\le\;
 \E^\P\big[U(X_T)\big].
 \e*
By the same argument, we see that the control process $\pi^*\in\Ac$ allows to reach the last upper bound $J_{\pi^*,\mu}(\pi^*)=\E^\P\big[U(X_T)\big]$, thus completing the proof.
\ep

{\color{black}

\section{Justification of the the approximate Nash equilibrium} \label{sec:approx-nash}

This section is devoted to the proof of Theorem \ref{thm:nash_equilibria}. We start by providing some properties on the coefficients $(\mathbf{B},\mathbf{\Sigma}).$ We recall that $\Pi:=\Pi^N$ is defined in \eqref{piN} and \eqref{PiN}. In the next Lemma, for any $\xb:=(x^1,\cdots,x^N) \in \R^N,$ we denote $m^N(\xb):=\frac{1}{N} \sum_{j=1}^N \delta_{x^j}.$

\begin{Lemma} \label{lemma:coefficients}
    $\mathbf{(i)}$ There exists a constant $C>0$ (independent of $N$) s.t. for all $t\in[0,T],\mathbf{x}, \mathbf{y} \in \R^N,$ and $\beta\in[0,1]^N$, we have
    \begin{align*}
        &\big| \mathbf{\Sigma}^{k,k} (t,\Pi^{-i}_t(\beta),\mathbf{x}) 
        -
        \mathbf{\Sigma}^{k,k} (t,\Pi_t,\mathbf{x})\big|
        +
        \big| \mathbf{B}^k (t,\Pi^{-i}_t(\beta),\mathbf{x}) 
        -
        \mathbf{B}^k (t,\Pi_t,\mathbf{x})\big| 
        \le \frac{C}{N}\;\;\mbox{for all }k\neq i
        \\
        &\big| \mathbf{\Sigma}^{k,k} (t,\Pi_t,\mathbf{x}) 
        -
        \mathbf{\Sigma}^{k,k} (t,\Pi_t,\mathbf{y})\big|
        +
        \big| \mathbf{B}^k (t,\Pi_t,\mathbf{x}) 
        -
        \mathbf{B}^k (t,\Pi_t,\mathbf{y})\big| 
        \le
        C \Big(\phi^k+\frac{1}{N} \sum_{j=1}^N \phi^j\Big)(t,\mathbf{x},\mathbf{y}),
    \end{align*}
  where $\phi^j(t,\mathbf{x},\mathbf{y}):=\big|(b,\sigma)(t,x^j,m^N(\mathbf{x}))-(b,\sigma)(t,y^j,m^N(\mathbf{y}))\big|$, and
    \begin{align*}
        &\sup_{1 \le q \neq e \le N}\big| \mathbf{\Sigma}^{e,q}(t,\Pi_t,\mathbf{x}) \Big|+\sup_{1 \le k \le N}\bigg|\mathbf{\Sigma}^{k,k}(t,\Pi_t,\mathbf{x}) - \frac{\sigma(t,x^k,m^N(\xb)) }{1+\pi^k_t}\bigg| \le \frac{C}{N},
        \\
        &\big| \mathbf{\Sigma}^{k,k} (t,\Pi_t,\mathbf{x}) 
        \big|
        +
        \big| \mathbf{B}^k (t,\Pi_t,\mathbf{x}) 
        \big| 
        \le
        C \bigg[ \big|(b,\sigma)(t,x^k,m^N(\mathbf{x}))\big| + \frac{1}{N} \sum_{j=1}^N \big|(b,\sigma)(t,x^j,m^N(\mathbf{x})) \big| \bigg].
    \end{align*}
\end{Lemma}
\begin{proof}
We organise the proof in three steps.
\\
{\it Step 1:} We first derive the estimates for the drift coefficients $\mathbf{B}^k$. For simplicity of notation, we denote $B^{i,k}_t:=  \mathbf{B}^k (t,\Pi^{-i}_t(\beta),\mathbf{x})$, $b^k_t:=b(t,x^k,m^N)$, $k=1,\cdots,N,$ and we recall that
$$
       B^{i,i}_t=\frac{1}{N} \!\sum_{ j \neq i}\! \beta^j_t B^{i,j}_t
        - \frac{1}{N} \!\sum_{j \neq i} \pi^i_t B^{i,i}_t+ b^i_t
        =
        \frac{1}{N} \!\sum_{j \neq i}\! \beta^j_t B^{i,j}_t
        \!-\! \frac{N\!-\!1}{N}\!\pi^i_t B^{i,i}_t \!+\! b^i_t.
$$
This provides
    \begin{eqnarray}\label{Bii}
        B^{i,i}_t
        =
        \frac{1}{r^i_t} \Big(\frac{1}{N} \sum_{j \neq i} \beta^j_t B^{i,j}_t+ b^i_t\Big),
        &\mbox{where}&
        r^j_t:= 1 + \frac{N-1}{N} \pi^j_t.
    \end{eqnarray}
Similarly, we have for  $k \neq i$ that
    \begin{align*}
        B^{i,k}_t&=\frac{1}{N} \sum_{j=1}^N \pi^j_t B^{i,j}_t
        - \frac{1}{N} \big(\beta^k_t B^{i,k}_t+ \sum_{j \neq i} \pi^k_t B^{i,k}_t \big) 
        + b^k_t\\
        &=\frac{1}{N} \sum_{j=1}^N \pi^j_t B^{i,j}_t
        -\frac{1}{N} \beta^k_t B^{i,k}_t - \frac{N-1}{N} \pi^k_t B^{i,k}_t + b^k_t,
        \\
        &=
        \frac{1}{N} \sum_{j \neq i} \pi^j_t B^{i,j}_t 
         +\frac{1}{N}  \frac{\pi^i_t}{r^i_t } \bigg[ \frac{1}{N} \sum_{j \neq i}  \beta^j_t B^{i,j}_t+ b^i_t  \bigg] 
        - \frac{\beta^k_t}{N} B^{i,k}_t - \frac{N-1}{N} \pi^k_t B^{i,k}_t + b^k_t,
    \end{align*}
    by substituting the expression of $B^{i,i}$ from \eqref{Bii}. Then, we get
    \begin{align}\label{Bik}
        B^{i,k}_t
        =
        \frac{A^{i}(t,\Pi^{-i}_t(\beta),\mathbf{x})
        + \frac{1}{N} \frac{\pi^i_t}{r^i_t }b^i_t \!+\! b^k_t}{r^k_t + \frac{\beta^k_t}{N}},
    \end{align}
where we denoted $A^{i}(t,\Pi^{-i}_t(\beta),\mathbf{x})
        :=
        \frac{1}{N} \sum_{j \neq i} 
        (\pi^j_t \!+\! \frac{\beta^j_t\pi^i_t}{Nr^i_t } ) B^{i,j}_t$. Multiplying the left hand side by $(\pi^k_t \!+\! \frac{\beta^k_t\pi^i_t}{Nr^i_t } )$ and summing over $k\neq i$, we recover the expression of $A^{i,N}$:
    \begin{eqnarray}\label{AiN}
        A^{i}(t,\Pi^{-i}_t(\beta),\mathbf{x})
        =
        \frac{\frac1N\sum_{j\neq i}a^{i,j}(b^j+\frac1N\frac{\pi^ib^i}{r^i})
               }
               {1-\frac1N\sum_{j\neq i}a^{i,j}},
       &\mbox{with}&
       a^{i,j}:=\frac{\pi^j_t \!+\! \frac{\beta^j_t\pi^i_t}{Nr^i_t }}{r^j+\frac{\beta^j}{N}}.
    \end{eqnarray}
From this expression, we directly see the existence of a constant $C$ independent of $(N,i,\Pi,\beta)$ such that for all $(\mathbf{x},\mathbf{y}) \in \R^N \x \R^N,$
    \begin{eqnarray*}
        \big|A^{i}(t,\Pi^{-i}_t(\beta),\mathbf{x}) - A^{i}(t,\Pi^{-i}_t(\beta),\mathbf{y}) \big| 
        &\le& 
        C  \frac{1}{N} \sum_{k=1}^N \big|b(t,x^k,m^N(\mathbf{x}))-b(t,y^k,m^N(\mathbf{y})) \big|
    \\
    \mbox{and}~~
    \big|A^{i}(t,\Pi^{-i}_t(\beta),\mathbf{x}) \big| 
    &\le& 
    C  \frac{1}{N} \sum_{k=1}^N \big|b(t,x^k,m^N(\mathbf{x})) \big|.
    \end{eqnarray*}
The corresponding estimates for $(B^{i,k})_{1 \le i, k \le N}$ are then immediately inherited by using the expressions \eqref{Bii} and \eqref{Bik}.

\vspace{3mm}
\noindent {\it Step 2.} We next estimate the difference $\mathbf{B}^k(t,\Pi^{-i}_t(\beta),\mathbf{x})-\mathbf{B}^k(t,\Pi_t,\mathbf{x}).$ Denoting $B^{k}_t:=\mathbf{B}^k(t,\Pi_t,\mathbf{x}),$ it follows from \eqref{Bii} and \eqref{Bik} that
    \begin{eqnarray*}
    B^{i,i}_t-B^{i}_t
    &=&
    \frac{1}{r^i_t} \frac{1}{N} \sum_{j \neq i} \Big((\beta^j_t-\pi^j_t)B^j_t 
                                                                          +\beta^j_t(B^{i,j}_t- B^j_t)\Big)
    \\
        B^{i,k}_t-B^{k}_t
        &=&
        \frac{A^{i}(t,\Pi^{-i}_t(\beta),\mathbf{x})-A^{i}(t,\Pi_t,\mathbf{x})}
               {r^k_t + \frac{\beta^k_t}{N}},
        ~~\mbox{for}~~k\neq i.
    \end{eqnarray*}
As $\pi$ and $\beta$ are valued in $[0,1]$, it follows from the explicit expression \eqref{AiN} that the exists a constant $C$ independent of $(N, i, \pi,\beta)$ such that
         $$
         |B^{i,k}_t-B^k_t| \le \frac{C}{N}\Big[1 +|B^{i,k}_t(\pi^i_t+\beta^k_t)| \Big],
         ~~\mbox{for all}~~k\neq i.
         $$
{\it Step 3.} We next focus on the estimates of the diffusion coefficient. Denote similarly
$\Sigma^{i,k,q}_t:=  \mathbf{\Sigma}^{k,q} (t,\Pi^{-i}_t(\overline{\beta}),\mathbf{x})$, $\sigma_t^k:=\sigma(t,x^k,m^N)$, and recall that
    \begin{eqnarray}\label{Sigma(i)}
    \Sigma^{i,i,q}_t
    &=&
    \frac{1}{N} \sum_{j \neq i}\beta^j_t\Sigma^{i,j,q}_t
        - \frac{1}{N} \sum_{j \neq i} \pi^i_t \Sigma^{i,i,q}_t+ \sigma_t^i \mathbf{1}_{\{q=i\}},
        \\
        \Sigma^{i,k,q}_t
        &=&
        \frac{1}{N} \sum_{j=1}^N \pi^j_t \Sigma^{i,j,q}_t
        - \frac{1}{N} \big( \sum_{j \neq i}^N \pi^k_t \Sigma^{i,k,q}_t + \beta^k_t \Sigma^{i,k,q}_t \big) + \sigma_t^k \mathbf{1}_{\{q=k\}},\;\;k \neq i,~1\le q\le N.
    \nonumber
    \end{eqnarray}
Similar calculations as in Step 1 provide
    \begin{eqnarray}\label{Sigma(ii)}
        \Sigma^{i,k,q}_t
        &=&
        \frac{Q^{i,q}(t,\Pi^{-i}_t(\beta),\mathbf{X}_t) 
                + \sigma_t^k\mathbf{1}_{\{q=k\}}
                +\frac1N\frac{\pi^i_t}{r^i_t} \sigma_t^i\mathbf{1}_{\{q=i\}}}
        {r^k_t + \frac{\beta^k_t}{N}}, ~~\mbox{for}~~k\neq q,
    \end{eqnarray}
    where
    \begin{eqnarray}\label{Sigma(iii)}
        Q^{i,q}(t,\Pi^{-i}_t(\beta),\mathbf{x})
        &=&
        \frac{\frac1N\sum_{j\neq i}a^{i,j}(\sigma_t^j\mathbf{1}_{\{q=j\}}
                +\frac1N\frac{\pi^i_t}{r^i_t}\sigma_t^i \mathbf{1}_{\{q=i\}})
               }
               {1-\frac1N\sum_{j\neq i}a^{i,j}},
    \end{eqnarray}
with the same $r^j$ and $a^{i,j}$ as those defined in Step 1. We then directly estimate that
    \begin{eqnarray*}
        \big|Q^{i,q}(t,\Pi^{-i}_t(\overline{\beta}),\mathbf{x}) - Q^{i,q}(t,\Pi^{-i}_t(\overline{\beta}),\mathbf{y}) \big|
        &\le&  \frac{C}{N} \big|\sigma(t,x^q,m^N(\mathbf{x}))\mathbf-\sigma(t,y^q,m^N(\mathbf{y})) \big| ,
        \\
        \big|Q^{i,q}(t,\Pi^{-i}_t(\overline{\beta}),\mathbf{x}) \big| 
        &\le&   \frac{C}{N}  \big|\sigma(t,x^q,m^N(\mathbf{x})) \big| ,
    \end{eqnarray*}
which induce the required corresponding estimates for the diffusion coefficients..
     
We finally estimate the difference $\mathbf{\Sigma}(t,\Pi^{-i}_t(\beta),\mathbf{x})-\mathbf{\Sigma}(t,\Pi_t,\mathbf{x})$ by following the same argument as in Step 2, based on the calculations of Step 3.    
\qed
\end{proof}

\vspace{3mm}
{\it In the following, we assume that the coefficients $(\sigma,\mbox{b})$ are bounded, in order to avoid unnecessary technical details.}

\vspace{3mm}

Let us consider the sequence of processes $(\beta^N)_{N \in \N^*}:=(\beta^{N,1},\cdots,\beta^{N,N})_{N \in \N^*}.$ Recall the $\P-$dynamics of the corresponding $i$--deviated equity process $(X^{i,1},\cdots,X^{i,N})$ satisfies: 
\begin{align*}
    \mathrm{d}X^{i,k}_t
    =
    \mathbf{B}^{k}(t,\Gamma_t^{-i}(\beta^N_t),\mathbf{X}^{i}_t) \mathrm{d}t
    &+
    \sum_{q\neq i}\mathbf{\Sigma}^{k,q}(t,\Gamma_t^{-i}(\beta^N_t),\mathbf{X}^{i}_t) \mathrm{d}W^q_t
    \\
    &+\mathbf{\Sigma}^{k,i}(t,\Gamma_t^{-i}(\beta^N_t),\mathbf{X}^{i}_t)(\mathrm{d}W^i_t-\psi^i_t \mathrm{d}t),~~k =1,\cdots,N.
\end{align*}

\begin{Lemma} \label{lemma:deviate_player}
For all $k=1,\cdots,N$, we have
    \begin{align*}
        \E^{\P} \bigg[ \sup_{t \in [0,T]} |X^{i,k}_t-X^k_t|^2 \bigg] 
        \le
        \frac{C}{N}.
    \end{align*}
    Consequently, denoting by $\mu^{i,N}:=\frac1N\sum_{j=1}^N\delta_{X^{i,j}}$, $\mu^N :=\frac1N\sum_{j=1}^N\delta_{X^{j}}$ the corresponding empirical measures, we have
    \begin{align*}
        \lim_{N \to \infty}\frac{1}{N} \sum_{i=1}^N \E^{\P}\bigg[ \sup_{t \in [0,T]} \big|X^i_t-X^{i,i}_t \big|^2 \bigg]
        =
        0\;\;\mbox{and}\;\;\lim_{N \to \infty} \frac{1}{N}\sum_{i=1}^N\E^{\P} \Big[\Wc_2 \big( \mu^{i,N}, \mu^N\big) \Big]=0.
    \end{align*}
\end{Lemma}

\begin{proof}
Denote $B^{i,k}_t(\mathbf{x}):=\mathbf{B}^{k}(t,\Gamma_t^{-i}(\beta^N),\mathbf{x})$, $B^{k}_t(\mathbf{x}):=\mathbf{B}^{k}(t,\Gamma_t,\mathbf{x})$, and similarly $\Sigma^{i,k,q}_t(\mathbf{x}):=\mathbf{\Sigma}^{k,q}(t,\Gamma_t^{-i}(\beta^N),\mathbf{x})$, $\Sigma^{k,q}_t(\mathbf{x}):=\mathbf{\Sigma}^{k,q}(t,\Gamma_t,\mathbf{x})$. Direct calculation provides   
    \begin{align*}
    \mathrm{d}(X^{i,i}_t-X^i_t)
        =&
        (B^{i,i}_t(\Xbb^i_t)-B^{i,i}_t(\Xbb_t)) \mathrm{d}t +\sum_{q=1}^N (\Sigma^{i,i,q}_t(\Xbb^i_t)-\Sigma^{i,q}_t(\Xbb_t)) \mathrm{d} W^q_t
        \\
        &-\frac{\Sigma^{i,i,i}_t(\Xbb^i_t)-\Sigma^{i,i,i}_t(\Xbb_t)}{\Sigma^{i,i,i}_t(\Xbb_t)} (B^{i,i}_t(\Xbb_t)-B^i_t(\Xbb_t)) \mathrm{d}t,
        \\
        \mathrm{d}(X^{i,k}_t-X^k_t)
        =&
        (B^{i,k}_t(\Xbb^i_t)-B^k_t(\Xbb_t)) \mathrm{d}t +\sum_{q=1}^N (\Sigma^{i,k,q}_t(\Xbb^i_t)-\Sigma^{k,q}_t(\Xbb_t)) \mathrm{d} W^q_t
        \\
        &-\Sigma^{i,k,i}_t(\Xbb^i_t)\frac{B^{i,i}_t(\Xbb_t)-B^i_t(\Xbb_t)}{\Sigma^{i,i,i}_t(\Xbb_t)} \mathrm{d}t,~~k \neq i.
    \end{align*}
By the properties proved in Lemma \ref{lemma:coefficients}, together with the Gronwall Lemma, we show the existence of a constant $C$ (independent of $N$) such that for any $k=1,\cdots,N$ and $t \in [0,T],$
    \begin{align*}
        \E^{\P} \bigg[ \sup_{s \in [0,t]} |X^{i,k}_s-X^k_s|^2 \bigg] 
        \le
        C \bigg[ \frac{1}{N} + \int_0^t  \E^{\P} \big[\Wc_2 (\mu^N_s, \mu^{i,N}_s)^2  \big] \mathrm{d}s  \bigg].
    \end{align*}
    It is then straightforward that
    \begin{align*}
        \E^\P \big[\Wc_2 (\mu^N_t, \mu^{i,N}_t)^2 \big]
        \le 
        \frac{1}{N} \sum_{k=1}^N \E^{\P} \big[ |X^{i,k}_t-X^k_t|^2 \big] \le
        C \bigg[ \frac{1}{N} + \int_0^t  \E^{\P} \big[\Wc_2 (\mu^N_s, \mu^{i,N}_s)^2  \big] \mathrm{d}s  \bigg],
    \end{align*}
which implies by the Gronwall Lemma the existence of a constant, still denoted $C$ for simplicity, such that
    \begin{align*}
        \sup_{t \in [0,T]}\E^\P \big[\Wc_2 (\mu^N_t, \mu^{i,N}_t)^2 \big]
        \le 
        \frac{C}{N},
    \end{align*}
which implies the remaining required claims.
    \qed
\end{proof}

\vspace{3mm}

For any $\mu \in \Pc_{\Sc}(\pi),$ we introduce the set
\begin{align*}
    \Sc(\pi,\mu)
    :=
    \big\{ \P^\beta_{\pi,\mu} \circ (X)^{-1},\;\; \beta \in \Ac  \big\}.
\end{align*}
and the sequence of probability distributions $(\mathrm{P}^N)_{N \in \N^*} \subset \Pc \big( \widehat{\Om} \x \Pc(\widehat{\Om}) \big)$
\begin{align*}
    \mathrm{P}^N
    :=
    \frac{1}{N} \sum_{i=1}^N \P^i_{\Pi,\beta^N} \circ \big( X^{i,i}, \mu^{i,N} \big)^{-1}.
\end{align*}
We denote by $(\widehat{X}, \widehat{\mu})$ the canonical variable on $\Pc \big( \widehat{\Om} \x \Pc(\widehat{\Om}) \big).$ 
\begin{Proposition}\label{prop:limit_property}
    The sequence $(\mathrm{P}^N)_{N \in \N^*}$ is relatively compact in $\Wc_2$, and for any limit point $\mathrm{P}^\infty$, we have
    \begin{align*}
        \muh \in \Pc_{\Sc}(\pi)\;\;\mbox{and}\;\;\Lc^{\mathrm{P}^\infty} (\widehat{X} | \muh ) \in \Sc(\pi,\muh),
        ~~\mathrm{P}^\infty\mbox{--a.e.}~\mbox{on}~ \widehat{\Om} \x \Pc(\widehat{\Om}).
    \end{align*}
\end{Proposition}

Before starting the proof of this Proposition, let us show briefly how we use this result to prove Theorem \ref{thm:nash_equilibria}.
\paragraph*{Proof of Theorem \ref{thm:nash_equilibria}} We set
\begin{align*}
    \varepsilon^N
    :=
    \sup_{\beta} J_i\big((\Pi^N)^{-i}(\beta)\big)
    -
    J_i(\Pi^N)
    =
    \sup_{\beta} J_1\big((\Pi^N)^{-1}(\beta)\big)
    -
    J_1(\Pi^N).
\end{align*}
Let $(\beta^N)_{N \in \N^*}$ be a sequence satisfying $ J_1((\Pi^N)^{-1}(\beta^N))-J_1(\Pi^N) \ge \varepsilon^N -2^{-N},$ for all $N \ge 1.$ By construction, $\varepsilon^N \ge 0$ for all $N \ge 1,$ and  
\begin{align*}
    \frac{1}{N} \sum_{i=1}^N J_i\big((\Pi^N)^{-i}(\beta^N)\big)
    -
    J_i(\Pi^N)
    &=
    \frac{1}{N} \sum_{i=1}^N \big( \E^{\P^i_{\Pi,\beta^N}}[U(X^{i,i}_T)]
    -
    \E^{\P}[U(X^{i}_T)] \big)
    \\
    &=
    \frac{1}{N} \sum_{i=1}^N \E^{\P^i_{\Pi,\beta^N}}[U(X^{i,i}_T)]
    -
    \E^{\P}[\langle U, \mu^N_T \rangle].
\end{align*}
By combining Lemma \ref{lemma:deviate_player}, this provides 
\begin{align*}
    &\limsup_{N \to \infty} \frac{1}{N} \sum_{i=1}^N
    \frac{1}{N} \sum_{i=1}^N J_i\big((\Pi^N)^{-i}(\beta^N)\big)
    -
    J_i(\Pi^N)
    \\
    &=
    \limsup_{N \to \infty} \frac{1}{N} \sum_{i=1}^N \big(\E^{\P^i_{\Pi,\beta^N}}[U(X^{i,i}_T)]
    -
    \E^{\P^i_{\Pi,\beta^N}}[\langle U, \mu^{i,N}_T \rangle] \big)
    \\
    &=
    \limsup_{N \to \infty} \big(\E^{\Pr^N}[U(\Xh_T)-\langle U, \muh_T \rangle] \big)
    \\
    &=
    \E^{\mathrm{P}^\infty}\bigg[ U(\widehat{X}_T) - \int_{\widehat{\Om}} U(x_T) \muh(\mathrm{d}x) \bigg]    =
    \E^{\mathrm{P}^\infty}\bigg[ \E^{\mathrm{P}^\infty}\big[ U(\widehat{X}_T) \big| \muh \big] - \int_{\widehat{\Om}} U(x_T) \muh(\mathrm{d}x) \bigg],
\end{align*}
where $\Pr^\infty$ is some limit of sequence $(\Pr^N)_{N \ge 1}$ in $\Wc_2$ sense, see Proposition \ref{prop:limit_property}. Here, we also used the fact that $U$ is continuous with linear growth. 
By Corollary \ref{cor:MFG_sol}, any element of $\Pc_{\Sc}(\pi)$ is a solution of the mean field game of mutual holding associated with control $\pi.$ Moreover, by Proposition \ref{prop:limit_property}, we have that for $\mathrm{P}^\infty$ a.e. $\om \in \widehat{\Om} \x \Pc(\widehat{\Om}),$ $\muh \in \Pc_{\Sc}(\pi)\;\;\mbox{and}\;\;\Lc^{\mathrm{P}^\infty} (\widehat{X} | \muh ) \in \Sc(\pi,\muh)$. Therefore
\begin{align*}
    \E^{\mathrm{P}^\infty}\big[ U(\widehat{X}_T) \big| \muh \big] - \int_{\widehat{\Om}} U(x_T) \muh(\mathrm{d}x) \le 0,\;\mathrm{P}^\infty\mbox{--a.s}.
\end{align*}
and we finally deduce from the definition of $\beta^N$ together with the non--negativity of $\varepsilon^N$ that 
\begin{eqnarray*}
0 \;\le\; \liminf_{N \to \infty} \varepsilon^N\;\le\; \limsup_{N \to \infty} \varepsilon^N  \;\le\;
    \limsup_{N \to \infty} \frac{1}{N} \sum_{i=1}^N J_i\big((\Pi^N)^{-i}(\beta^N)\big)-J_i(\Pi^N) &\le& 0,
\end{eqnarray*}
which induces the required result.
\qed

\paragraph*{Proof of Proposition \ref{prop:limit_property}}
    We recall that $\frac{\mathrm{d}\P^i_{\Pi,\beta^N}}{\mathrm{d}\P}= Z^i_T.$ Therefore, in order to obtain the asymptotics of the distributions $\Lc^{\P^i_{\Pi,\overline{\beta}^N}}(X^{i,i},\mu^{i,N})$, we focus on the distributions $\Lc^{\P}(Z^i,X^{i,i},\mu^{i,N}).$ Using Lemma \ref{lemma:deviate_player}, one can verify that
    \begin{align} \label{eq:same_limit}
        \lim_{N \to \infty} \Wc_2 \Big( \frac{1}{N} \sum_{i=1}^N \Lc^{\P}(Z^i,X^{i,i},\mu^{i,N}), \frac{1}{N} \sum_{i=1}^N \Lc^{\P}(Z^i,X^{i},\mu^{N}) \Big) =0.
    \end{align}
Consequently, we may instead focus on the sequence $\big(\frac{1}{N} \sum_{i=1}^N \Lc^{\P}(Z^i,X^{i},\mu^{N}) \big)_{N \ge 1}.$ Recall the dynamics of $(Z^i,X^i)$: 
    \begin{eqnarray}
        \frac{\mathrm{d} Z^i_t}{Z^i_t }
        &=&
        \frac{\mathbf{B}^{i}(t,\Pi^{-i}_t(\beta^N),\mathbf{X}_t)-\mathbf{B}^{i}(t,\Pi_t,\mathbf{X}_t)}
                        {\mathbf{\Sigma}^{i,i}(t,\Pi^{-i}_t(\beta^N),\mathbf{X}_t)} \mathrm{d}W^i_t
        \label{Xi}\\
        \mathrm{d}X^i_t
        &=&
        B(t,X^i_t,\mu^N_t) \mathrm{d}t
        +
        \sum_{j=1}^N\Sigma^{i,j}(t,X^i_t,\mu^N_t) \mathrm{d}W^j_t.
        \label{Zi}
    \end{eqnarray}
{\it Step 1: Coefficients asymptotics.} By \eqref{eq:def_coef} and \eqref{def:nash_equilibria}, notice that $\mathbf{B}^{i}(t,\Pi_t,\mathbf{X}_t)= B(t,X^i_t,\mu^N_t)$. Then denoting 
 \begin{eqnarray}\label{qN}
 q^N_t(\mathrm{d}u,\mathrm{d}x)
        &:=&
        \frac{1}{N} \sum_{i=1}^N \delta_{(\beta^{N,j}_t,X^j_t)}(\mathrm{d}u,\mathrm{d}x),
\end{eqnarray}
it follows from \eqref{Bii} that
        \begin{align*}
        &\Big( 1 + \frac{N-1}{N} \pi^i_t \Big)\mathbf{B}^{i}(t,\Pi^{-i}_t(\beta^N),\mathbf{X}_t)
        \\
        &=
        \frac{1}{N} \sum_{j\neq i} \beta^{N,j}_t \mathbf{B}^j(t,\Pi^{-i}_t(\beta^N),\mathbf{X}_t) + b(t,X^i_t,\mu^N_t) 
        \\
        &=
        \frac{1}{N} \sum_{j\neq i} \beta^{N,j}_t B(t,X^j_t,\mu^N_t) + b(t,X^i_t,\mu^N_t)
        +\frac{1}{N} \sum_{j\neq i} \beta^{N,j}_t (\mathbf{B}^j(t,\Pi^{-i}_t(\beta^N),\mathbf{X}_t)-\mathbf{B}^j(t,\Pi_t,\mathbf{X}_t))
        \\
        &=
        \int u B(t,x,\mu^N_t) q^N_t(\mathrm{d}u,\mathrm{d}x) + b(t,X^i_t,\mu^N_t)
        +\frac{1}{N} \sum_{j\neq i} \beta^{N,j}_t (\mathbf{B}^j(t,\Pi^{-i}_t(\beta^N),\mathbf{X}_t)-\mathbf{B}^j(t,\Pi_t,\mathbf{X}_t)),
    \end{align*}
Then, it follows from Lemma \ref{lemma:coefficients} that
    \begin{eqnarray} \label{eq:delta}
        \delta_N
        \;:=\;
        \Big|\mathbf{B}^{i}(t,\Pi^{-i}_t(\beta),\mathbf{X}_t)
        -\frac{\int u B(t,x,\mu^N_t) q^N_t(\mathrm{d}u,\mathrm{d}x) + b(t,X^i_t,\mu^N_t)}
                {1 + \pi(t,X^i_t,\mu^N_t)}
        \Big|
        &\stackrel{N \to \infty}{\longrightarrow}& 0.~~~~~
    \end{eqnarray}
Similarly, it follows from \eqref{Sigma(i)}-\eqref{Sigma(ii)}-\eqref{Sigma(iii)} together with Lemma \ref{lemma:coefficients} that
    \begin{equation} \label{eq:r}
        r_N
        :=
        \sup_{1 \le i \le N}\bigg|\mathbf{\Sigma}^{i,i}(t,\Pi^{-i}_t(\beta^N),\mathbf{X}_t)
        -
        \frac{\sigma(t,X^i_t,\mu^N) }{1+\pi(t,X^i_t,\mu^N_t)} \bigg| 
        + \sup_{q \neq j}
        \Big|\Sigma^{q,j}(t,X^i_t,\mu^N_t) \Big|
        \stackrel{N \to \infty}{\longrightarrow} 0.
    \end{equation}
 \noindent {\it Step 2: Formulation of the martingale problem.} On a filtered probability space $(\Om,\F,\P),$ for a continuous function $m: [0,T] \longrightarrow \Pc(\R),$ let $\Mc[\Om,\F,\P](m)$ be the set of $\Pc(\Pc([0,1] \x \R))$--valued $\F$--predictable processes $\Lambda:=(\Lambda_t)_{t \in [0,T]}$ satisfying 
    $$
        \Lambda_t \big( \{ q \in \Pc([0,1] \x \R):\;\;q([0,1],\mathrm{d}x)=m(\mathrm{d}x) \} \big)=1\;\;\mathrm{d}t \otimes \mathrm{d}\P\mbox{--a.e.}
    $$
    The set $\Mc[\Om,\F,\P](m)$ has to be seen as a set of controls. Now, let $(X,Z,W,M, \Lambda)$ be variables such that: 
    \\
    $\bullet$ $\Lambda \in \Mc[\Om,\F,\P](\mu),$ 
    \\
    $\bullet$ $M(\mathrm{d}q,\mathrm{d}t)$ is a martingale measure with quadratic variation $\Lambda$, so that $W_t\!=\!M\big(\Pc(\Pc([0,1] \!\x\! \R)) \x [0,t]\big)$, $t\in[0,T]$, defines a Brownian motion  (see \citeauthor*{el1990martingale} \cite{el1990martingale} for an overview of martingale measure), 
    \\
    $\bullet$ and $(X,Z,W)$ is a weak solution of 
    \begin{eqnarray*}
        \mathrm{d} X_t
        &\!\!\!=&\!\!\!\!
        B(t,X_t,\mu_t) \mathrm{d}t
        + \frac{\sigma(t,X_t,\mu_t)}{1+\pi(t,X_t,\mu_t)} \mathrm{d}W_t,
        ~\mbox{with}~\mu_t=\Lc(X_t),~t\in[0,T],
        \\
        \frac{\mathrm{d}Z_t}{Z_t}
        &\!\!\!=&\!\!\!\!
         \int \frac{1+\pi(t,X_t,\mu_t)}{\sigma(t,X_t,\mu_t)}\bigg[\frac{\int u B(t,x,\mu_t) q(\mathrm{d}u,\mathrm{d}x) + b(t,X_t,\mu_t)}{1 + \pi(t,X_t,\mu_t)}  - B(t,X_t,\mu_t)\bigg] M(\mathrm{d}q,\mathrm{d}t).
    \end{eqnarray*}
In order to formulate the martingale problem associated with processes $(X,Z,W,\Lambda)$, we introduce the generator $\Lc=\Lc_s^{\mu_s,q}$ of $(X,Z,W)$. Then, $(X,Z,W, \Lambda)$ is a weak solution of the last SDE if for any $f \in C_b^2(\R \x \R \x \R),$ the process
    \begin{align*}
        M^f_t(X,Z,W,\Lambda,\mu):= f(X_t,Z_t,W_t)-\int_0^t \int\Lc_s^{\mu_s,q} f(X_s,Z_s,W_s) \Lambda_s(\mathrm{d}q) \mathrm{d}s,~t \in [0,T],
    \end{align*}
is a $(\F,\P)$--martingale.

 \medskip
 \noindent {\it Step 3: Identification of the limit.} We now show that all the possible limits of our sequence are related to a distribution of type $\P\circ(X, Z, W, \Lambda)^{-1}$. Let $E:=\Pc([0,1] \x \R)$, and $\M(E)$ the space of all Borel measures $q( \mathrm{d}t,  \mathrm{d}e)$ on $[0,T] \x E$, whose marginal distribution on $[0,T]$ is the Lebesgue measure $ \mathrm{d}t$, i.e. $q( \mathrm{d}t, \mathrm{d}e)=q_t(\mathrm{d}e) \mathrm{d}t$ for some family $(q_t)_{t \in [0,T]}$ of Borel probability measures on $E$.
	We denote by $\Lambda$ the canonical element of $\M(E)$ and set
	\begin{align*}
	\Lambda_{t \wedge \cdot}(\mathrm{d}s, \mathrm{d}e) :=  \Lambda(\mathrm{d}s, \mathrm{d}e) \big|_{ [0,t] \x E} + \delta_{e_0}(\mathrm{d}e) \mathrm{d}s \big|_{(t,T] \x E},\; \text{for some fixed $e_0 \in E$.}
	\end{align*}
Let $\widetilde{\Om}:=C([0,T];\R^3) \x \M(E) \x \Pc(\Cc)$,  and denote the corresponding  canonical process and canonical filtration by $(\widetilde{X},\widetilde{Z},\widetilde{W},\widetilde{\Lambda},\widetilde{\mu})$ and $\widetilde{\F}:=(\widetilde{\Fc}_t)_{t \in [0,T]}$, $\Fct_t:=\sigma \{\Xt_{t \wedge \cdot},\Zt_{t \wedge \cdot}, \Wt_{t \wedge \cdot}, \widetilde{\Lambda}_{t \wedge \cdot}, \mut \circ (\Xh_{t \wedge \cdot})^{-1} \}$. We shall also denote by $(\mub,\mu)$ the canonical variable on $\Pc(\widetilde{\Om}) \x \Pc(\Cc).$ 

Recall the random measures $q^N$ defined in \eqref{qN}, and let $(\widehat{\mathrm{P}}^N)_{N \in \N^*} \subset \Pc\big( \Pc(\widetilde{\Om}) \x \Pc(\Cc) \big)$ be the sequence defined by
    \begin{align*}
        \widehat{\mathrm{P}}^N
        :=
        \frac{1}{N} \sum_{i=1}^N \Lc^{\P}(\mub^N,\mu^{N})\;\;\mbox{where}\;\;\mub^N
        :=
        \frac{1}{N} \sum_{i=1}^N \delta_{(X^i,Z^i,W^i,\Lambda^N,\mu^N)},
        ~\mbox{and}~
        \Lambda^N:=\delta_{q^N_t}(\mathrm{d}q)\mathrm{d}t.
    \end{align*}
Under Assumption \ref{assum:bsigma}, together with with the boundedness of $(b,\sigma)$, we see that $(B^i,\Sigma^{k,q})$ are also bounded uniformly in $N$, and we easy verify that $(\widehat{\mathrm{P}}^N)_{N \in \N^*}$ is relatively compact in $\Wc_2.$ Let $\widehat{\mathrm{P}}^\infty$ be the limit of a sub--sequence. For simplicity, we use the same notation for the sequence and its sub--sequence. We now show that for $\widehat{\mathrm{P}}^\infty$ a.e. $\om \in \Pc(\widetilde{\Om}) \x \Pc(\Cc)$: 
    \begin{align} \label{eq:martingale_property}
        \big(M^f_t(\widetilde{X},\widetilde{Z},\widetilde{W},\widetilde{\Lambda},\widetilde{\mu}) \big)_{t \in [0,T]}\;\;\mbox{is a}\;(\widetilde{\F},\mub(\om))\mbox{--martingale for all }f\in C_b^2(\R^3).
    \end{align}
To prove this result, we introduce $M^{f,i}$ the martingale associated to $(X^i,Z^i,W^i)$ whose dynamics are recalled in \eqref{Xi}-\eqref{Zi}:
    \begin{align*}
        M^{f,i}_t
        :=&
        f(X^i_t,Z^i_t,W^i_t)
        -
        \bigg[\int_0^t \nabla_x f(X^i_s,Z^i_s,W^i_s) B(s,X^i_s,\mu^N_s) \mathrm{d}s + \frac{1}{2} \nabla_x^2 f(X^i_s,Z^i_s,W^i_s) \mathrm{d} \langle X^i \rangle_s
        \\
        &+
        \frac{1}{2} \nabla_z^2 f(X^i_s,Z^i_s,W^i_s) \mathrm{d} \langle Z^i \rangle_s
        +
        \frac{1}{2} \nabla_w^2 f(X^i_s,Z^i_s,W^i_s) \mathrm{d}s
        +
        \nabla_{xz}^2 f(X^i_s,Z^i_s,W^i_s) \mathrm{d} \langle X^i,Z^i \rangle_s
        \\
        &+\nabla_{xw}^2 f(X^i_s,Z^i_s,W^i_s) \mathrm{d} \langle X^i,W^i \rangle_s
        +
        \nabla_{wz}^2 f(X^i_s,Z^i_s,W^i_s) \mathrm{d} \langle W^i,Z^i \rangle_s \bigg].
    \end{align*}
In order to handle the singularity apprearing in the diffusion of $X^i$ in \eqref{Xi}, we introduce the subsets of $\Omt$ 
    $$
    \begin{array}{cc}
        C_1(t)
        :=
        \Big\{ \omt \in \Omt:\;\mbox{the map}\;\omt' \longmapsto \mathbf{1}_{\{B(t,\Xt_t(\omt'),\mut_t(\omt')) \ge 0\}}\;\mbox{is continuous at the point }\omt \Big\},
        \\
    C_2(t)
        :=
        \Big\{ \omt \in \Omt:\;\mbox{the map}\;\omt' \longmapsto M^f_t(\Xt,\Zt,\Wt,\widetilde{\Lambda},\mut)(\omt')\;\mbox{is continuous at the point }\omt
        \Big\}.
    \end{array}
    $$
Denoting $\mub_t:=\mub \circ (\Xt_t)^{-1}$, $t \in [0,T],$ it follows from Proposition \ref{prop:lebesgue} that for $\widehat{\Pr}^\infty$--a.e. $\om,$ $\mub_t(\om)(\mathrm{d}x)\mathrm{d}t$ has a density w.r.t. the Lebesgue measure on $\R \x [0,T].$ As $B(t,x,m) \ge 0$ is equivalent to $b(t,x,m)+c(t,m) \ge 0$, we obtain by Assumption \ref{assum:bsigma} that
    \begin{align} \label{cond:portemanteau}
    \mub ( C_1(t) )=1,\;\mathrm{d}\widehat{\Pr}^\infty \otimes \mathrm{d}t-\mbox{a.e. and therefore}
    ~~\mub ( C_2(t) )=1,\;\mathrm{d}\widehat{\Pr}^\infty \otimes \mathrm{d}t-\mbox{a.e.} 
    \end{align}
    Let $\Phi: \Omt \longrightarrow \R$ be a bounded continuous function.  Recall from \eqref{eq:delta}-\eqref{eq:r} that $ (\delta_N,r_N)\longrightarrow (0,0)$ and, from Lemma \ref{lemma:deviate_player}, $\sup_{1 \le q \neq e \le N}\big| \Sigma^{e,q}(t,\Pi_t,\mathbf{x}) \Big| \le \frac{C}{N}$. Then, in view of \eqref{cond:portemanteau}, it follows from the Portemanteau Theorem, together with the proof of Proposition \ref{prop:lebesgue}, that
    \begin{align*}
        &\E^{\widehat{\Pr}^\infty} \Big[ \E^{\mub} \Big[\Big( M^f_t(\Xt,\Zt,\Wt,\widetilde{\Lambda},\mut) - M^f_s(\Xt,\Zt,\Wt,\widetilde{\Lambda},\mut) \Big) \Phi \big( \Xt_{s \wedge \cdot},\Zt_{s \wedge \cdot},\Wt_{s \wedge \cdot},\widetilde{\Lambda}_{s \wedge \cdot},\mut_{s \wedge \cdot} \big) \Big]^2 \Big]
        \\
        &=
        \lim_{N \to \infty} \E^{\P} \Big[ \Big| \frac{1}{N} \sum_{i=1}^N \Big( M^{f,i}_t - M^{f,i}_s \Big) \Phi \big( X^i_{s \wedge \cdot},Z^i_{s \wedge \cdot},W^i_{s \wedge \cdot},\Lambda^N_{s \wedge \cdot},\mu^N_{s \wedge \cdot} \big) \Big|^2 \Big].
    \end{align*}
    Notice that
    \begin{align*}
        M^{f,i}_\cdot= f(X^i_0,Z^i_0,W^i_0) &+ 
        \int_0^\cdot \nabla_x f(X^i_t,Z^i_t,W^i_t) \sum_{j=1}^N \Sigma^{i,j}(t,X^i_t,\mu^N_t) \mathrm{d}W^j_t 
        \\
        &+ \int_0^\cdot \nabla_z f(X^i_t,Z^i_t,W^i_t) \mathrm{d}Z^i_t 
        + \nabla_w f(X^i_t,Z^i_t,W^i_t) \mathrm{d}W^i_t.
    \end{align*}
    As $(W^i)_{1 \le i \le N}$ are independent Brownian motions and $\sup_{1 \le q \neq e \le N}\big| \Sigma^{e,q}(t,\Pi_t,\mathbf{x}) \Big| \le \frac{C}{N},$ we deduce by a similar argument as in \citeauthor*{lacker2017limit} \cite[Proof of Proposition 5.1.]{lacker2017limit} and \citeauthor*{djete2019general} \cite[Proof of Proposition 4.17]{djete2019general} that
    \begin{align*}
        \lim_{N \to \infty} \E^{\P} \Big[ \Big| \frac{1}{N} \sum_{i=1}^N \Big( M^{f,i}_t - M^{f,i}_s \Big) \Phi \big( X^i_{s \wedge \cdot},Z^i_{s \wedge \cdot},W^i_{s \wedge \cdot},\Lambda^N_{s \wedge \cdot},\mu^N_{s \wedge \cdot} \big) \Big|^2 \Big]=0.
    \end{align*}
    By taking a countable sequence of $f \in C^2_b(\R^3),$ we deduce that \eqref{eq:martingale_property} is verified. By similar techniques as \citeauthor*{MFD-2020} \cite{MFD-2020}, we find that ,
    \begin{align} \label{eq:mu}
        \mub(\om)\Big[\mut=\mu(\om)=\Lc^{\mub(\om)}(\Xt) \Big]=1,
        ~\mbox{for}~ \widehat{\Pr}^\infty-\mbox{a.e.}~\om \in \Pc(\Omt) \x \Pc(\Cc),
     \end{align}
and 
    $$
        \widetilde{\Lambda}_t \big( \big\{ q \in \Pc([0,1] \x \R):q([0,1],\mathrm{d}x)=\mut_t(\mathrm{d}x) \big\} \big)=1,
        \mathrm{d}t \otimes \mathrm{d}\mub(\om)-\mbox{a.e.}
    $$ 
which means that $\widetilde{\Lambda} \in \Mc[\Omt,\Ft,\mub(\om)]((\mut_t)_{t \in [0,T]})$, $\mathrm{d}t \otimes \mathrm{d}\mub(\om)-$a.e. 
Consequently, by Step 1, for $\om \in \Pc(\Omt) \x \Pc(\Cc),$  on $\big(\Omt \x [0,1],(\Fct_t \otimes \Bc([0,1]))_{t \in [0,T]}, \mub(\om) \otimes \lambda \big)$ an extension of $\big(\Omt,(\Fct_t)_{t \in [0,T]}, \mub(\om) \big),$ there exists a martingale measure $\widetilde{M}(\mathrm{d}q,\mathrm{d}t)$ with quadratic variation $\widetilde{\Lambda}$ such that for $\widehat{\Pr}^\infty$--a.e. $\om,$ the process $(\Xt,\Zt,\Wt,\widetilde{\Lambda})$ satisfies $\mut_t=\Lc^{\mub(\om)}(\Xt_t),$ $\widetilde{\Lambda} \in \Mc[\Omt,\Ft,\mub(\om)]((\mut_t)_{t \in [0,T]}),$
    \begin{eqnarray*}
        \mathrm{d} \Xt_t
        &\!\!=&\!\!
        B(t,\Xt_t,\mut_t) \mathrm{d}t
        + \frac{\sigma(t,\Xt_t,\mut_t)}{1+\pi(t,\Xt_t,\mut_t)} \mathrm{d}\Wt_t,\;\;\Wt_t=\widetilde{M}(E \x [0,t]),
        \\
        \frac{\mathrm{d}\Zt_t}{\Zt_t}
        &\!\!=&\!\!
         \int \frac{1+\pi(t,\Xt_t,\mut_t)}{\sigma(t,\Xt_t,\mut_t)}\bigg[\frac{\int u B(t,x,\mut_t) q(\mathrm{d}u,\mathrm{d}x) + b(t,\Xt_t,\mut_t)}{1 + \pi(t,\Xt_t,\mut_t)}  - B(t,\Xt_t,\mut_t)\bigg] \widetilde{M}(\mathrm{d}q,\mathrm{d}t).
    \end{eqnarray*}
In view of the last dynamics of $\Xt$, this shows that $\mu(\om) \in \Pc_{\Sc}(\pi)$, for $\widehat{\Pr}^\infty$--a.e. $\om$. 
 
 \medskip
 \noindent {\it Step 4:} For all $\om \in \Pc(\Omt) \x \Pc(\Cc),$ let us define the probability
    \begin{align*}
        \frac{\mathrm{d}\mub^\circ(\om)}{\mathrm{d}\mub(\om)}:=\Zt_T
    \end{align*}
    and the process $(\widetilde{N}_t)_{t \in [0,T]}$ by
    $$
        \widetilde{N}_\cdot
        :=
        \int_0^\cdot \int \frac{1+\pi(t,\Xt_t,\mut_t)}{\sigma(t,\Xt_t,\mut_t)}\bigg[\frac{\int u B(t,x,\mut_t) q(\mathrm{d}u,\mathrm{d}x) + b(t,\Xt_t,\mut_t)}{1 + \pi(t,\Xt_t,\mut_t)}  - B(t,\Xt_t,\mut_t)\bigg] \widetilde{M}(\mathrm{d}q,\mathrm{d}t)
    $$
    By Girsanov Theorem, $\Wt^\circ_\cdot:=\Wt_\cdot- \langle \Wt, \widetilde{N} \rangle_\cdot$ is a $\mub^\circ(\om)$--Brownian motion. It is straightforward that
    \begin{align*}
        \mathrm{d} \Xt_t
        =
        \frac{\int u B(t,x,\mu_t(\om)) q(x)(\mathrm{d}u)\mu_t(\om)(\mathrm{d}x) \widetilde{\Lambda}_t(\mathrm{d}q) + b(t,\Xt_t,\mu_t(\om))}{1 + \pi(t,\Xt_t,\mu_t(\om))} \mathrm{d}t
        +
        \frac{\sigma(t,\Xt_t,\mu_t(\om))}{1+\pi(t,\Xt_t,\mu_t(\om))} \mathrm{d}\Wt^\circ_t,
    \end{align*}
    where $\R \ni x \longmapsto q(x)(\mathrm{d}u) \in \Pc([0,1])$ is a Borel map s.t. $q(\mathrm{d}u,\mathrm{d}x)=q(x)(\mathrm{d}u)q([0,1],\mathrm{d}x).$
    Using some Markovian projection techniques, we can verify that $\Lc^{\mub^\circ(\om)}(\Xt) \in \Sc(\pi,\mu(\om)).$ By combining all the results, for any bounded continuous function $\Phi,$ we get that
    \begin{align*}
        \Lim_{N \to \infty}\frac{1}{N} \sum_{i=1}^N \E^{\P^i_{\Pi,\beta^N}} \big[ \Phi (X^{i,i},\mu^{i,N}) \big]
        &=
        \Lim_{N \to \infty}\frac{1}{N} \sum_{i=1}^N \E^{\P} \big[ Z^i_T \Phi (X^{i,i},\mu^{i,N}) \big]
        \\
        &= \E^{\widehat{\Pr}^{\infty}} \big[ \E^{\mub} \big[ \Zt_T \Phi (\Xt,\mu) \big] \big]
        = \E^{\widehat{\Pr}^{\infty}} \big[ \E^{\mub^\circ} \big[ \Phi (\Xt,\mu) \big] \big].    \end{align*}
    Therefore,
    \begin{align*}
         \Lim_{N \to \infty}\frac{1}{N} \sum_{i=1}^N \P^i_{\Pi,\beta^N} \circ \big( X^{i,i}, \mu^{i,N} \big)^{-1}
         =
        \widehat{\Qr}^\infty
        :=\int_{\Pc(\Omt) \x \Pc(\Cc)}\mub^\circ(\om) \big( \Xt, \mut \big)^{-1}  \widehat{\Pr}^\infty(\mathrm{d}\om).
    \end{align*}
    Notice that $\widehat{\Qr}^\infty \in \Pc(\Omh \x \Pc(\Omh)).$
    Let $Q^\infty \in \Pc(\Omt )$ be defined by
    \begin{align*}
        Q^\infty
        :=
        \int_{\Pc(\Omt) \x \Pc(\Cc)}\mub^\circ(\om) \big( \Xt,\Wt^\circ,\Zt,\Lambda, \mut \big)^{-1}  \widehat{\Pr}^\infty(\mathrm{d}\om).
    \end{align*}
    As $\widehat{\Pr}^\infty$--a.s. $\om,$ $\mut=\mu(\om),$ $\mub(\om)$ a.e.,  $\Wt^\circ$ and $\mut$ are $Q^\infty$--independent. Therefore, the conditional dynamic of $\Xt$ given $\mut$ under $Q^\infty$ is 
    \begin{align*}
        \mathrm{d} \Xt_t
        =
        \frac{\int u B(t,x,\mut_t(\om)) q(x)(\mathrm{d}u)\mut_t(\om)(\mathrm{d}x) \widetilde{\Lambda}_t(\mathrm{d}q) + b(t,\Xt_t,\mut_t(\om))}{1 + \pi(t,X_t,\mut_t(\om))} \mathrm{d}t
        +
        \frac{\sigma(t,\Xt_t,\mut_t(\om))}{1+\pi(t,X_t,\mut_t(\om))} \mathrm{d}\Wt^\circ_t.
    \end{align*}
 This shows that $\Lc^{Q^\infty}(\Xt | \mut) \in \Sc(\pi,\mut)$ $Q^\infty$--a.e. and, since $\widehat{\Qr}^\infty=\Lc^{Q^\infty}(\Xt,\mut),$ it follows that $\Lc^{\widehat{\Qr}^\infty}(\Xh | \muh) \in \Sc(\pi,\muh)$ $\widehat{\Qr}^\infty$--a.s.

    \qed

}
\section{Existence for the mean field SDE of mutual holding}
\label{sec:SDE}

This section provides a wellposedness result for a class of McKean-Vlasov 
SDEs taking values in $\R^d$, for some fixed dimension $d \in \N^*$, with
singular diffusion coefficient. We consider the general form
\begin{equation} \label{eq:weak_McK-irregular}
	    X_t
	    =
	    X_0
	    +
	    \int_0^t \Lambda(\Theta_s) \mathrm{d}s
	    +
	    \int_0^t \Gamma(\Theta_s) \Phi \big(\gamma(\Theta_s) \big) \mathrm{d}W_s,
	    t\in[0,T],~\mbox{with}~
	    \Theta_s=(s,X_s,\mu_s),
\end{equation}
for some given $\Fc_0-$measurable r.v. $X_0$ with values in $\R^d$, where 
$\mu_s=\Lc(X_s)$ denotes the law of $X_s,$ and $W$ is a $\R^d$--valued $\F$--Brownian 
motion on a filtered probability space $(\Om,\F,\P)$. 

Our main concern is about the irregularity of the diffusion coefficient which is illustrated here by the general function
  $$\Phi:\R^d\longrightarrow \S^d,~~\mbox{Borel measurable bounded},$$
where $\S_d$ denote the collection of all $d\x d$-–dimensional matrices
with real entries. 

Our restrictions on the coefficients $\Lambda,\Gamma,\gamma$ allow to cover our application to the mutual holding problem.

\begin{Assumption}\label{assum:SDE}
The functions $(\Lambda,\Gamma,\gamma): (t,x,m)\in [0,T] \x \R^d \x \Pc(\R^d) 
\longrightarrow \R^d \x \S_d \x  \R^d$ are Borel maps with quadratic growth in $(x,m)$ uniformly in $t$ satisfying in addition: 
\begin{enumerate}

\item[{\rm (i)}] The diffusion coefficient is uniformly elliptic, i.e. $\inf_{\theta,u\neq 0}\frac{u^\intercal(\Gamma\Phi\circ\gamma)(\theta) u}{|u|^2}>0;$

\item[{\rm (ii)}] $\Lambda,\Gamma,\gamma$ are uniformly Lipschitz in $m$, i.e. there exists $C >0,$ such that 
	\begin{eqnarray*}
	    \sup_{[0,T] \x \R^d} \big| (\Lambda,\Gamma,\gamma)(.,m)
			            -(\Lambda,\Gamma,\gamma)(.,m')
			   \big |
	\;\le\;
	C\;\Wc_2 (m,m'),
	&\mbox{for all}&
	m,m'\in\Pc(\R^d);
	\end{eqnarray*}
\item[{\rm (iii)}] {\color{black} For all $m \in \Pc(\R),$ and Lebesgue-a.e. $t \in [0,T],$
\begin{align*}
    &\Jc(t,m)
    :=
    \big\{ 
        x \in \R: \Phi \circ \gamma (t,\cdot)\;\mbox{is continuous at the point}\;(x,m)
    \big\}
\end{align*}
has full Lebesgue measure, i.e. $\Jc(s,m)^c$ is Lebesgue--negligible.
}
\end{enumerate}
\end{Assumption}

\begin{Definition}\label{def:SDE}
A weak solution of the SDE \eqref{eq:weak_McK-irregular} is a six-tuple $(\widetilde{\Om},\widetilde{\Fc},\widetilde{\F},\widetilde{\P},\widetilde{W},\widetilde{X})$ where $(\widetilde{\Om},\widetilde{\Fc},\widetilde{\F},\widetilde{\P})$ is a filtered probability space, supporting an $\widetilde{\F}$--Brownian motion $\widetilde{W}$,  and $\widetilde{X}$ is a $\widetilde{\F}$--adapted continuous process satisfying the dynamics \eqref{eq:weak_McK-irregular}.
\end{Definition}
	
\begin{Theorem} \label{thm:weak_existence}

.
Under Assumption \ref{assum:SDE}, for all $\mu_0\in\Pc_{p}(\R^d)$ with $p>2$, the SDE \eqref{eq:weak_McK-irregular} has at least one weak solution $(\widetilde{\Om},\widetilde{\Fc},\widetilde{\F},\widetilde{\P},\widetilde{W},\widetilde{X})$ satisfying $\E^{\widetilde{\P}} \big[ \sup_{t \in [0,T]} |\widetilde{X}_t|^p \big] < \infty.$
\end{Theorem}
	
\proof
{\color{black}
For $n \in \N^*$ and $t \in [0,T),$ we introduce the $n-$dyadic projection $[t]^n:=2^{-n}T\lfloor\frac{2^nt}{T}\rfloor,$, where $\lfloor\cdot\rfloor$ denotes the floor function. Consider the Euler discretization of the SDE \eqref{eq:weak_McK-irregular}:
\begin{align*}
        X^n_\cdot
	    =
	    X_0
	    \!+\!\!
	    \int_0^\cdot\! \Lambda(\Theta^n_{[t]^n}\!) \mathrm{d}t
	    \!+\!\!
	    \int_0^\cdot\! (\Gamma\, \Phi \!\circ\!\gamma)\big(\Theta^n_{[t]^n}\!\big) \mathrm{d}W_t~\mbox{with}~
	    \Theta^n_t=(t,X^n_t,\mu^n_t),~\mu^n_t:=\Lc(X^n_t),
\end{align*}
{\bf 1.} As $\mu_0 \in \Pc_p(\R^d)$ with $p>2$, the sequence $\big(\mu^n:=\Lc(X^n)\big)_{n \in \N^*}$ is relatively compact in $\Wc_2$, and converges to some limit $\mu^\infty$, in $\Wc_2$, after possibly passing to a sub--sequence. Then
 $$
    \Pr^n(\mathrm{d}m,\mathrm{d}x,\mathrm{d}t)
    :=
    \delta_{\mu^n_t}(\mathrm{d}m) \mu^n_t(\mathrm{d}x) \frac{\mathrm{d}t}{T} 
    ~\longrightarrow~
    \Pr^\infty(\mathrm{d}m,\mathrm{d}x,\mathrm{d}t)
    :=
    \delta_{\mu^\infty_t}(\mathrm{d}m) \mu^\infty_t(\mathrm{d}x) \frac{\mathrm{d}t}{T},
    ~\mbox{in}~\Wc_2.
 $$
Moreover, as $\E \big[ \sup_{t \in [0,T]} |X^n_t - X^n_{[t]^n}|^2 \big]=0$, as $n\to\infty$, we see that
\begin{eqnarray} \label{eq:eq:delay_limit}
    \lim_{n \to \infty} \Wc_2 \big(\overline{\Pr}^n, \Pr^n \big)=0,
    &\mbox{where}&
    \overline{\Pr}^n(\mathrm{d}m,\mathrm{d}x,\mathrm{d}t)
    :=
    \delta_{\mu^n_{[t]^n}}(\mathrm{d}m) \mu^n_{[t]^n}(\mathrm{d}x) \frac{\mathrm{d}t}{T}.
\end{eqnarray}
Since the diffusion coefficient is uniformly elliptic, it follows from similar techniques as in the proof of Proposition \ref{prop:lebesgue} that $\mu^\infty$ is absolutely continuous w.r.t. the Lebesgue measure on $[0,T]\x\R$. Together with Assumption \ref{assum:SDE} (iii), this implies that the set of discontinuity points of the map $\Gamma\,\Phi\circ\gamma$ is $\Pr^\infty-$negligible. Then, it follows from the Portemanteau Theorem that, for all bounded continuous function $\varphi$:
\begin{align} \label{eq:portemanteau_appendix}
    \!\!\lim_{n \to \infty} \!\int\!\!\varphi(x)(\Gamma\Phi\!\circ\!\gamma)(t,x,m)^2  
                                          \;\Pr^n\!(\mathrm{d}m,\!\mathrm{d}x,\!\mathrm{d}t) 
    \!= \!
    \int\!\!\varphi(x)(\Gamma\Phi\!\circ\!\gamma)(t,x,m)^2 
             \;\Pr^\infty\!(\mathrm{d}m,\!\mathrm{d}x,\!\mathrm{d}t).
\end{align}
{\bf 2.} For an arbitrary function $f \in C_b^2(\R^d),$ and $t\in[0,T]$, if follows from It\^{o}'s formula that
\begin{align*}
    \langle f,\mu^\infty_t \rangle &=
    \lim_{n \to \infty} \E[f(X^n_t)]
    \\
    &=\lim_{n \to \infty}\;\;
	        \E[f(X_0)]
	        +
	        \int_0^t \E\bigg[\nabla f(X^n_s)\!\cdot\!\Lambda(\Theta^n_{[s]^n})             +\frac{1}{2} \mbox{{\rm Tr}}\big(\nabla^2\!f(X^n_s)\Phi \circ\gamma(\Theta^n_{[s]^n}) \big) 
	                        \bigg]
	        \mathrm{d}s
	        \\
    &=\lim_{n \to \infty}\;\;
	        \E[f(X_0)]
	        +
	        \int_0^t \E\bigg[\nabla f(X^n_{[s]^n})\!\cdot\!\Lambda(\Theta^n_{[s]^n})             +\frac{1}{2} \mbox{{\rm Tr}}\big(\nabla^2\!f(X^n_{[s]^n})\Phi \circ\gamma(\Theta^n_{[s]^n}) \big) 
	                        \bigg]
	        \mathrm{d}s
\end{align*}
by \eqref{eq:eq:delay_limit}. Then, it follows from \eqref{eq:portemanteau_appendix} that
	\begin{align*}
    \langle f,\mu^\infty_t \rangle 
    &=\langle f,\mu_0 \rangle +
	       \int_0^t \!\!\int_{\R^d}\; \Big(\nabla f(x)\!\cdot\!\Lambda(s,x,\mu^\infty_s)                        + \frac{1}{2} \mbox{{\rm Tr}}\big( \nabla^2f(x) \Phi \circ\gamma(s,x,\mu^\infty_s) \big) \Big) \; \mu^\infty_s(\mathrm{d}x)
	        \mathrm{d}s.
\end{align*}
From the arbitrariness of $f \in C_b^2(\R^d),$ and $t\in[0,T]$, we can find a filtered probability space $(\widetilde\Omega,\widetilde\F,\widetilde\Fc,\widetilde{\P})$ supporting a $\widetilde\F$--Brownian motion $\widetilde W$ and a $\widetilde\F$--adapted continuous process $\widetilde X$ where $\widetilde X$ satisfies \eqref{eq:weak_McK-irregular} with Brownian motion $\widetilde W,$ $\Lc^{\widetilde{\P}}(\widetilde X)=\mu^\infty,$ and by Fatou's Lemma $\E^{\widetilde{\P}} \big[ \sup_{t \in [0,T]} |\widetilde{X}_t|^p \big] < \infty$.
}
\ep

{\color{black}

\section{Appendix: Existence of density for the limiting distribution of a uniformly elliptic particles system} \label{sec:density}
In this last section, we show that, under appropriate conditions on the coefficients, the marginals of the limit distribution of some particles systems are absolutely continuous w.r.t. the Lebesgue measure. Before stating our result, let us first introduce the framework.
On the probability space $(\Om, \F, \P)$ supporting a sequence of Brownian motions $(W^i)_{i \ge 1},$ $\Wbb^N:=(W^1,\cdots,W^N),$ let $S^N:=(S^{i,N})_{1 \le i \le N}$, $N\ge 1$, be a sequence of processes satisfying
\begin{align*}
    S^{N}_\cdot
    =
    S^{N}_0
    +
    \int_0^\cdot p^{N}_t \mathrm{d}t
    +
    \int_0^\cdot q^{N}_t \mathrm{d}\Wbb^N_t
\end{align*}
with $\sup_{N \ge 1} \sup_{1 \le i \le N} |q^{i,i,N}_t|^2 \ge \delta,$ $\mathrm{d}\P \otimes \mathrm{d}t$--a.e. and for some $p >2,$
\begin{align} \label{cond:integrability}
    \sup_{N \ge 1}\frac{1}{N}\sum_{i=1}^N\E \bigg[ |S^{i,N}_0|^p
    +
    \int_0^T |p^{i,N}_t|^p \mathrm{d}t
    +
    \int_0^T |q^{i,i,N}_t|^p
    + N |\sum_{k\neq i} \sum_{j \neq i} q^{i,k,N}_t q^{i,j,N}_t|^2\mathrm{d}t
    \bigg]
    <
    \infty.
\end{align}
We denote by $\mu$ the canonical process on $\Pc(\Cc)$, and we introduce the sequence of laws of the particles empirical measure $(\Qr^N)_{N \in \N^\star}\subset \Pc(\Pc(\Cc))$ defined by
\begin{align*}
    \Qr^N
    :=
    \P \circ ( \mu^N )^{-1},
    ~~ \mbox{where}~~\mu^N:=\frac1N \sum_{i=1}^N \delta_{S^{i,N}}.
\end{align*}
The main result of this section is the following.

\begin{Proposition} \label{prop:lebesgue}
    The sequence $(\Qr^N)_{N \in \N^\star}$ is relatively compact in $\Wc_2.$ In addition, for any limit point $\Qr^\infty$, we have that
    \begin{align*}
    \mu_t(\mathrm{d}x)\mathrm{d}t\mbox{ is absolutely continuous w.r.t. the Lebesgue measure on }\R \x [0,T],
        ~\Qr^\infty\mbox{--a.s.}
    \end{align*}
\end{Proposition}

\begin{proof} Consider the sequence $(\widetilde{\Qr}^N)_{N \in \N^\star} \subset \Pc(\Pc(\Cc \x \M(\R) \x \M(\R)))$  defined by $\widetilde{\Qr}^N
        :=
        \P \circ ( \mut^N )^{-1}$ where
    \begin{align*}
       \mut^N:=\frac{1}{N} \sum_{i=1}^N \delta_{(S^{i,N},\widetilde{p}^N, \widetilde{a}^N)},\;\widetilde{b}^N:=\frac{1}{N} \sum_{i=1}^N \delta_{p^{i,N}_t}(\mathrm{d}e)\mathrm{d}t\;\mbox{and}\;\widetilde{a}^N:=\frac{1}{N} \sum_{i=1}^N \delta_{|q^{i,i,N}_t|^2}(\mathrm{d}e)\mathrm{d}t.
    \end{align*}
    Using \eqref{cond:integrability}, it is easy to see that $(\widetilde{\Qr}^N)_{N \in \N^\star}$ is relatively compact in $\Wc_2.$ Let $\widetilde{\Qr}^\infty$ be the limit of a sub--sequence. For simplicity, we use the same notation for the sequence and the sub--sequence.
    Let $\mut$ be the canonical variable on $\Pc(\Omt)$ and $(\Xt,\widetilde{p},\widetilde{a})$ be the canonical process on $\Omt:=\Cc \x \M(\R) \x \M(\R).$ We denote by $\Ft$ the canonical space on $\Omt.$ For any $f \in C_b^2(\R),$ on the canonical space $\Omt,$ we define the process
    \begin{align*}
      M^f(t,\Xt,\widetilde{b},\widetilde{a})
      :=
      f(\Xt_t)-f(\Xt_0)-\int_0^t \Big(\int_\R\; f'(\Xt_s) e \widetilde{p}_s(\mathrm{d}e) + \frac{1}{2} f^{''}(\Xt_s) e \widetilde{a}_s(\mathrm{d}e)\Big)\mathrm{d}s,
    \end{align*}
and we now show that $\widetilde{\Qr}^\infty$--a.e. $(M^f(t,\Xt,\widetilde{p},\widetilde{a}))_{t \in [0,T]}$ is a $(\Ft,\mut)$--martingale for all $f \in C_b^2(\R).$ For this, we introduce the process $(M^i_t)_{t \in [0,T]}$ defined by
 \begin{align*}
        M^i_t
        &:=
        f(X^i_t)-f(X^i_0)
        -\bigg[ \int_0^t f'(X^i_s) p^{i,N}_s \mathrm{d}s + \frac{1}{2} f^{''}(X^i_s) \sum_{j=1}^N \sum_{k=1}^N q^{i,j}_s q^{i,k}_s \mathrm{d}s \bigg]
        \\
        &=:
        \widetilde{M}^i_t -\int_0^t \frac{1}{2} f^{''}(X^i_s) \sum_{j=1,j \neq i}^N \sum_{k=1, k \neq i}^N q^{i,j}_s q^{i,k}_s \mathrm{d}s.
    \end{align*}
By It\^o's formula, we have 
    \begin{align*}
        M^i_t
        =
        \int_0^t f'(X^i_s) \sum_{j=1}^N q^{i,j,N}_s \mathrm{d}W^j_s.
    \end{align*}
    Next, let $\phi: \Omt \to \R$ be a continuous bounded function and $t \ge s,$ by using the weak convergence, \eqref{cond:integrability} and the independence of Brownian motions, we see that
    \begin{align*}
        &\E^{\widetilde{\Qr}^\infty} \Big[ \E^{\mut} \Big[\big( M^f(t,\Xt,\widetilde{p},\widetilde{a}) - M^f(s,\Xt,\widetilde{p},\widetilde{a}) \big) \phi(\Xt_{s \wedge \cdot},\widetilde{p}_{s \wedge \cdot},\widetilde{a}_{s \wedge \cdot}) \Big]^2 \Big]
        \\
        &=
        \Lim_{N \to \infty} \E^{\widetilde{\Qr}^N} \Big[ \E^{\mut} \Big[\big( M^f(t,\Xt,\widetilde{p},\widetilde{a}) - M^f(s,\Xt,\widetilde{p},\widetilde{a}) \big) \phi(\Xt_{s \wedge \cdot},\widetilde{p}_{s \wedge \cdot},\widetilde{a}_{s \wedge \cdot}) \Big]^2 \Big]
        \\
        &\le C \;
        \Lim_{N \to \infty} \E \bigg[ \bigg|\frac{1}{N} \sum_{i=1}^N\big( \widetilde{M}^i_t - \widetilde{M}^i_s \big) \phi(X^i_{s \wedge \cdot},\widetilde{p}^N_{s \wedge \cdot},\widetilde{a}^N_{s \wedge \cdot}) \bigg|^2 \bigg] 
        + \E \bigg[ \bigg|\frac{1}{N} \sum_{i=1}^N \int_s^t \sum_{j \neq i} \sum_{k \neq i} \big|q^{i,j}_s q^{i,k}_s \big| \mathrm{d}s  \bigg|^2 \bigg]
        \\
        &\le C \;
        \Lim_{N \to \infty} \E \bigg[ \bigg|\frac{1}{N} \sum_{i=1}^N\big( \widetilde{M}^i_t - \widetilde{M}^i_s \big) \phi(X^i_{s \wedge \cdot},\widetilde{p}^N_{s \wedge \cdot},\widetilde{a}^N_{s \wedge \cdot}) \bigg|^2 \bigg] 
        + \E \bigg[ \frac{1}{N} \sum_{i=1}^N \int_s^t \big|\sum_{j \neq i} \sum_{k \neq i} q^{i,j}_s q^{i,k}_s \big|^2 \mathrm{d}s   \bigg]
        \\
        &=0,
    \end{align*}
by \eqref{cond:integrability}. As the last calculation holds for countable set of $f \in C_b^2(\R),$ it follows that $\widetilde{\Qr}^\infty$--a.e. $(M^f(t,\Xt,\widetilde{p},\widetilde{a}))_{t \in [0,T]}$ is a $(\Ft,\mut)$--martingale for all $f \in C_b^2(\R).$ Therefore, for all $\omt \in \Omt,$ on $(\Omt \x [0,1],(\Fct_t \otimes \Bc([0,1]))_{t \in [0,T]},\mut(\om) \otimes \lambda)$ an extension of $(\Omt,(\Fct_t)_{t \in [0,T]},\mut(\om))$ supporting a Brownian motion $\Wt,$ we get that $\widetilde{\Qr}^\infty$--a.e. $\om,$ $\mut(\om)$--a.e.
    \begin{align*}
        \mathrm{d}\Xt_t
        =
        \int_\R e \widetilde{b}_t(\mathrm{d}e) \mathrm{d}t
        +
        \sqrt{\int_\R e \widetilde{a}_t(\mathrm{d}e)} \mathrm{d}\Wt_t.
    \end{align*}
    It is straightforward to check that $\int_\R e \widetilde{a}_t(\mathrm{d}e) \ge \delta$ $\mathrm{d}\mut(\om) \otimes \mathrm{d}t$ a.e. {\color{black} By \cite[Corollary 6.3.2.]{FK-PL-equations}, we can therefore conclude that $\mut(\om)\circ(\Xt_t)^{-1}(\mathrm{d}x)\mathrm{d}t$ has a density w.r.t. the Lebesgue measure on $\R \x [0,T].$ }
    \qed
\end{proof}

}

\bibliographystyle{plain}
\bibliography{Djete-Touzi-Rev_Arxiv}

\begin{thebibliography}{23}
\providecommand{\natexlab}[1]{#1}
\providecommand{\url}[1]{\texttt{#1}}
\expandafter\ifx\csname urlstyle\endcsname\relax
  \providecommand{\doi}[1]{doi: #1}\else
  \providecommand{\doi}{doi: \begingroup \urlstyle{rm}\Url}\fi

\bibitem[Acemoglu et~al.(2015)Acemoglu, Ozdaglar, and
  Tahbaz-Salehi]{acemoglu2015systemic}
D.~Acemoglu, A.~Ozdaglar, and A.~Tahbaz-Salehi.
\newblock Systemic risk and stability in financial networks.
\newblock \emph{American Economic Review}, 105\penalty0 (2):\penalty0 564--608,
  2015.

\bibitem[Aikman et~al.(2019)Aikman, Chichkanov, Douglas, Georgiev, Howat, and
  King]{aikman2019system}
D.~Aikman, P.~Chichkanov, G.~Douglas, Y.~Georgiev, J.~Howat, and B.~King.
\newblock System-wide stress simulation.
\newblock 2019.

\bibitem[Allen and Gale(2000)]{allen2000financial}
F.~Allen and D.~Gale.
\newblock Financial contagion.
\newblock \emph{Journal of political economy}, 108\penalty0 (1):\penalty0
  1--33, 2000.

\bibitem[Bayraktar et~al.(2020)Bayraktar, Guo, Tang, and
  Zhang]{bayraktar2020mckean}
E.~Bayraktar, G.~Guo, W.~Tang, and Y.~Zhang.
\newblock Mckean-vlasov equations involving hitting times: blow-ups and global
  solvability.
\newblock \emph{arXiv preprint arXiv:2010.14646}, 2020.

\bibitem[Bogachev et~al.()Bogachev, Krylov, R\"{o}ckner, and
  Shaposhnikov]{FK-PL-equations}
V.~I. Bogachev, N.~V. Krylov, M.~R\"{o}ckner, and S.~V. Shaposhnikov.
\newblock \emph{{F}okker–{P}lanck–{K}olmogorov {E}quations}.
\newblock Mathematical Surveys and Monographs. American Mathematical Society.

\bibitem[Carmona and Delarue(2018)]{carmona2018probabilisticI}
R.~Carmona and F.~Delarue.
\newblock \emph{Probabilistic theory of mean field games with applications
  {I}}, volume~83 of \emph{Probability theory and stochastic modelling}.
\newblock Springer International Publishing, 2018.

\bibitem[Carmona et~al.(2013)Carmona, Fouque, and Sun]{carmona2013mean}
R.~Carmona, J.-P. Fouque, and L.-H. Sun.
\newblock Mean field games and systemic risk.
\newblock \emph{Available at SSRN 2307814}, 2013.

\bibitem[Djete(2020)]{MFD-2020}
M.~F. Djete.
\newblock Extended mean field control problem: a propagation of chaos result.
\newblock \emph{arXiv preprint arXiv:2006.12996}, 2020.

\bibitem[Djete et~al.(2020)Djete, Possama{\"\i}, and Tan]{djete2019general}
M.~F. Djete, D.~Possama{\"\i}, and X.~Tan.
\newblock Mc{K}ean--{V}lasov optimal control: limit theory and equivalence
  between different formulations.
\newblock \emph{arXiv preprint arXiv:2001.00925}, 2020.

\bibitem[Eisenberg and Noe(2001)]{eisenberg2001systemic}
L.~Eisenberg and T.~H. Noe.
\newblock Systemic risk in financial systems.
\newblock \emph{Management Science}, 47\penalty0 (2):\penalty0 236--249, 2001.

\bibitem[El~Karoui and M{\'e}l{\'e}ard(1990)]{el1990martingale}
N.~El~Karoui and S.~M{\'e}l{\'e}ard.
\newblock Martingale measures and stochastic calculus.
\newblock \emph{Probability Theory and Related Fields}, 84\penalty0
  (1):\penalty0 83--101, 1990.

\bibitem[Garnier et~al.(2013)Garnier, Papanicolaou, and Yang]{garnier2013large}
J.~Garnier, G.~Papanicolaou, and T.-W. Yang.
\newblock Large deviations for a mean field model of systemic risk.
\newblock \emph{SIAM Journal on Financial Mathematics}, 4\penalty0
  (1):\penalty0 151--184, 2013.

\bibitem[Giesecke and Weber(2004)]{giesecke2004cyclical}
K.~Giesecke and S.~Weber.
\newblock Cyclical correlations, credit contagion, and portfolio losses.
\newblock \emph{Journal of Banking \& Finance}, 28\penalty0 (12):\penalty0
  3009--3036, 2004.

\bibitem[Hambly and S{\o}jmark(2019)]{hambly2019spde}
B.~Hambly and A.~S{\o}jmark.
\newblock An spde model for systemic risk with endogenous contagion.
\newblock \emph{Finance and Stochastics}, 23\penalty0 (3):\penalty0 535--594,
  2019.

\bibitem[Huang et~al.(2003)Huang, Caines, and Malham{\'e}]{huang2003individual}
M.~Huang, P.~Caines, and R.~Malham{\'e}.
\newblock Individual and mass behaviour in large population stochastic wireless
  power control problems: centralized and {N}ash equilibrium solutions.
\newblock In C.~Abdallah and F.~Lewis, editors, \emph{Proceedings of the 42nd
  IEEE conference on decision and control, 2003.}, pages 98--103. IEEE, 2003.

\bibitem[Huang et~al.(2006)Huang, Malham{\'e}, and Caines]{huang2006large}
M.~Huang, R.~Malham{\'e}, and P.~Caines.
\newblock Large population stochastic dynamic games: closed--loop
  {M}c{K}ean--{V}lasov systems and the {N}ash certainty equivalence principle.
\newblock \emph{Communications in Information \& Systems}, 6\penalty0
  (3):\penalty0 221--252, 2006.

\bibitem[Krylov(1980)]{KrylovControlledDiffusion}
N.~V. Krylov.
\newblock \emph{Controlled Diffusion Processes}.
\newblock Springer, 1980.

\bibitem[Lacker(2017)]{lacker2017limit}
D.~Lacker.
\newblock Limit theory for controlled {M}c{K}ean--{V}lasov dynamics.
\newblock \emph{SIAM Journal on Control and Optimization}, 55\penalty0
  (3):\penalty0 1641--1672, 2017.

\bibitem[Lasry and Lions(2007)]{lasry2007mean}
J.-M. Lasry and P.-L. Lions.
\newblock Mean field games.
\newblock \emph{Japanese Journal of Mathematics}, 2\penalty0 (1):\penalty0
  229--260, 2007.

\bibitem[Nadtochiy and Shkolnikov(2019)]{nadtochiy2019particle}
S.~Nadtochiy and M.~Shkolnikov.
\newblock Particle systems with singular interaction through hitting times:
  application in systemic risk modeling.
\newblock \emph{Ann. Appl. Probab.}, 29:\penalty0 89--129, 2019.

\bibitem[Nualart(2006)]{nualart2006malliavin}
D.~Nualart.
\newblock \emph{The Malliavin calculus and related topics}, volume 1995.
\newblock Springer, 2006.

\bibitem[Pardoux and Peng(1990)]{pardoux1990adapted}
E.~Pardoux and S.~Peng.
\newblock Adapted solution of a backward stochastic differential equation.
\newblock \emph{System and Control Letters}, 14\penalty0 (1):\penalty0 55--61,
  1990.

\bibitem[Shin(2009)]{shin2009securitisation}
H.~S. Shin.
\newblock Securitisation and financial stability.
\newblock \emph{The Economic Journal}, 119\penalty0 (536):\penalty0 309--332,
  2009.

\end{thebibliography}

\end{document}